\documentclass[reqno,11pt]{amsart}
\numberwithin{equation}{section}

\usepackage[utf8]{inputenc}
\usepackage[italian,english]{babel}
\usepackage{amsmath}
\usepackage{amsfonts}
\usepackage{amssymb,esint}
\usepackage{tikz}
\usepackage{geometry}
\usepackage{color}
\usepackage{mathrsfs}
\usepackage{mathtools}

\mathtoolsset{showonlyrefs}
\usepackage[colorlinks, citecolor=citegreen, linkcolor=refred]{hyperref}
\definecolor{citegreen}{rgb}{0,0.4,0}
\definecolor{refred}{rgb}{0.5,0,0}
\theoremstyle{plain}
\newtheorem {theorem}{Theorem}[section]
\newtheorem {lemma}[theorem]{Lemma}
\newtheorem {proposition} [theorem]{Proposition}
\newtheorem {corollary} [theorem]{Corollary}

\newtheorem{definition}[theorem]{Definition}
\newtheorem{remark}[theorem]{Remark}

\theoremstyle{remark}

\DeclarePairedDelimiter\abs{\lvert}{\rvert}
\DeclarePairedDelimiter\bigabs{\Big\lvert}{\Big\rvert}

\newcommand{\R}{\mathbb R}
\newcommand{\N}{\mathbb N}

\renewcommand{\theta}{\vartheta}

\newcommand{\crit}{{\rm Crit}}

\newcommand{\barint}
{\rule[.036in]{.12in}{.009in}\kern-.16in \displaystyle\int}

\newcommand{\dive}{{\mathrm{div}}}

\newcommand{\ep}{\varepsilon}
\newcommand{\pa}{\partial }

\usepackage{color}

\newcommand{\numberset}{\mathbb}
\renewcommand{\N}{\numberset{N}}
\renewcommand{\R}{\numberset{R}}
\newcommand{\Sf}{\numberset{S}}

\newcommand{\ric}{\mathop {\rm Ric}\nolimits}

\newcommand{\capa}{{\rm Cap}}
\newcommand{\D}{{\rm D}}
\newcommand{\dd}{{\,\rm d}}
\newcommand{\HH}{{\rm H}}
\newcommand{\hh}{{\rm h}}
\newcommand{\RR}{{\rm R}}

\newcommand{\na}{\nabla}

\newcommand{\Om}{\Omega}
\renewcommand{\phi}{\varphi}
\renewcommand{\epsilon}{\varepsilon}
\title{Minkowski Inequalities via Nonlinear Potential Theory}
\author[V.~Agostiniani]{Virginia Agostiniani}
\address{V.~Agostiniani, Universit\`a degli Studi di Verona,
strada le Grazie 15, 37134 Verona, Italy}
\email{virginia.agostiniani@univr.it}
\author[M.~Fogagnolo]{Mattia Fogagnolo}
\address{M.~Fogagnolo, Universit\`a degli Studi di Trento,
via Sommarive 14, 38123 Povo (TN), Italy}
\email{mattia.fogagnolo@unitn.it}

\author[L.~Mazzieri]{Lorenzo Mazzieri}
\address{L.~Mazzieri, Universit\`a degli Studi di Trento,
via Sommarive 14, 38123 Povo (TN), Italy}
\email{lorenzo.mazzieri@unitn.it}

\begin{document}

\maketitle

%%%%%%%%%%%%%%%%%%%%%%%%%%%%%%%%%%%%%%%%%%%%%%%%%%

\begin{abstract}
In this paper, we prove an extended version of the Minkowski Inequality, holding for any smooth bounded set $\Omega \subset \R^n$, $n\geq 3$. 
Our proof relies on the discovery of {\em effective monotonicity formulas} holding along the level set flow of the $p$-capacitary potentials associated with $\Om$, for every $p$ sufficiently close to $1$. Besides constituting a neat improvement of
those introduced in~\cite{Fog_Maz_Pin}
to treat the case of convex domains,
these formulas
testify the existence of a link between the monotonicity formulas derived by Colding and Minicozzi for the level set flow of Green's functions and the monotonicity formulas employed by Huisken, Ilmanen and several other authors in studying the geometric implications of the Inverse Mean Curvature Flow. 
In dimension $n\geq 8$, our conclusions are stronger than the ones obtained so far through the latter mentioned technique.
\end{abstract}

% appunto alcune MSC tra cui scegliere
% (e forse altre andrebbero aggiunte)
% 35B06 (PDE - symmetries, invariants of pdes)
% 53C21 (methods of Riem Geom (including pdes method))
% 39B62 (Functional inequalities, including subadditivity, convexity...)  
% 53C44 (Geometric evolution equations (mean curvature flow, Ricci flow, etc.))
% 35N25 (overdetermined bvp)
% 31C12 (Potential theory on Riemannian manifolds)
% 31C15 (Potentials and capacities)
% 49Q10 (Optimization of shapes other than minimal surfaces)

\bigskip
\noindent\textsc{MSC (2010): 
31C15, 53C44, 53C21, 35B06, 49Q10, 39B62. 
%49Q10 	Optimization of shapes other than minimal surfaces
%39B62  Functional inequalities, including subadditivity, convexity, etc. 
}


%%%%%%%%%%%%%%%%%%%%%%%%%%%%%%%%%%%%%%%%%%%%%%%%%%

%%%%%%%%%%%%%%%%%%%%%%%%%%%%%%%%%%%%%%%%%%%%%%%%%%
%%%%%%%%%%%%%%%%%%%%%%%%%%%%%%%%%%%%%%%%%%%%%%%%%%

\section{Introduction and statements of the main results}

%%%%%%%%%%%%%%%%%%%%%%%%%%%%%%%%%%%%%%%%%%%%%%%%%%
%%%%%%%%%%%%%%%%%%%%%%%%%%%%%%%%%%%%%%%%%%%%%%%%%%

A classical result in the theory of convex hypersurfaces in Euclidean spaces is the so called Minkowski inequality~\cite{Min}, which says that if $\Omega \Subset \R^n$,
$n\geq3$,
is a convex domain with smooth boundary and $\HH$ is the mean curvature of $\pa \Om$ computed with respect to the outward
unit normal, then 
\begin{equation}
\label{eq:mink}
\left(\frac{|\Sf^{n-1}|}{|\pa \Omega|} \right)^{\!\!1/(n-1)} \!\!\!\!\! \leq \,\,\,\,\, \fint\limits_{\pa \Omega}\frac{\HH}{n-1}  \,\, \dd \sigma \, ,
\end{equation}
with equality if and only if $\Omega$ is a ball. In other words, the inverse of the surface radius is a sharp lower bound for the averaged total mean curvature of $\pa \Om$. Observe that the above inequality can be conveniently rephrased as
\begin{equation}
\label{eq:mink2}
\left(\frac{|\pa \Omega|}{|\Sf^{n-1}|}\right)^{\!\!\frac{n-2}{n-1}} \!\!\! \leq \, \frac {1}{|\Sf^{n-1}|} \,\, \int\limits_{ \pa \Omega}   \frac{\HH}{n-1}  \, \dd \sigma \, ,
%(n-1) \, |\Sf^{n-1}|^{\frac{1}{n-1}} \,\, \leq \,   |\pa \Omega|^{-\frac{n-2}{n-1}}\int\limits_{ \pa \Omega}  {\HH}  \,\, \dd \sigma \, ,
\end{equation}
so that it can be combined with the standard Isoperimetric Inequality to deduce its volumetric version, also known in the literature as a higher order Isoperimetric Inequality (see~\cite{Chang_Wang_2013} and~\cite{qiu}) 
\begin{equation}
\label{eq:minkvol}
\left(\frac{| \Omega|}{|\mathbb{B}^{n}|}\right)^{\!\!\!\frac{n-2}{n}} \!\!\! \leq \, \frac {1}{|\Sf^{n-1}|} \,\, \int\limits_{ \pa \Omega}   \frac{\HH}{n-1}  \, \dd \sigma \, ,
\end{equation}
at least when $\Omega$ varies in the class of convex domains. It is worth recalling that both the Isoperimetric Inequality and the Minkowski Inequality are part of a family of inequalities involving quermassintegrals that were originally deduced in the context of convex analysis from the classical Aleksandrov-Fenchel mixed volume inequalities~\cite{Ale_1937, Ale_1938,fenchel}. A natural question, raised by several authors (see~\cite{Tru,Hui_video, Chang_Wang_2011,Chang_Wang_2013}), is whether the Minkowski Inequality~\eqref{eq:mink2} as well as its volumetric version~\eqref{eq:minkvol} hold true for larger classes of domains than just for the convex one.

\medskip

Positive answers to these questions have been proposed so far using the Inverse Mean Curvature Flow (IMCF from now on) and methods based on Optimal Transport. 
Our main concern in this paper is to propose an alternative technology based on Nonlinear Potential Theory, which is powerful enough to recover, improve and extend all the so far known results on these topics. Surprisingly, this new approach provides simplified arguments, which are also very flexible and likely to be exportable to several interesting frameworks, such as complete
manifolds with nonnegative Ricci curvature and asymptotically flat manifolds with nonnegative scalar curvature.
In Section~\ref{sec:method}, we will describe in more details the main features of this approach, drawing a systematic comparison with the existing curvature flow techniques. Here, we just anticipate that the cornerstone of our method is the discovery of {\em effective monotonicity formulas} (see Theorem~\ref{monotone1} and Theorem~\ref{monotone2}), holding along the level sets of the $p$-capacitary potential $u_p : \R^n \setminus \overline{\Om} \rightarrow \R$ associated with $\Omega$. Besides their geometric implications, these formulas have a technical relevance on their own, as they persist through all the possible singularities of the flow. It is worth noticing that, in the present framework, the flow singularities correspond to the critical points of $u_p$, and these might in principle be arranged in sets of full measure. This means that, albeit the level set flow is possibly subject to jumps, our monotonicity formulas are strong enough to survive them. Finally, from a theoretical point of view, these formulas can be seen as the crucial step towards the completion of a program initiated in the series of works~\cite{Ago_Maz_3,Ago_Maz_2, Ago_Fog_Maz, Fog_Maz_Pin}
and intended to link the monotonicity formulas employed by Huisken, Ilmanen and other authors in studying the geometric implications of the IMCF (see e.g.,~\cite{Hui_video, Hui_Ilm, Guan_Li, freire, wei_kottler, Wei,mccormick,Brendle,ge-hyperbolic1,ge-hyperbolic2,ge-kottler,deLima-hyperbolic,Bra_Mia,Bra_Nev}
to the monotonicity formulas discovered by Colding and Minicozzi in~\cite{Colding_1,Colding_Minicozzi, Colding_Minicozzi_2} for the level set flow of the Green's functions on complete manifolds with nonnegative Ricci curvature.
%From a theoretical point of view, these formulas testify the existence of a link between the well known monotonicity formulas for the IMCF (see for example the one reported in~\eqref{eq:mon}, but also the Geroch monotonicity of the Hawking Mass) and the monotonicity formulas derived by Colding and Minicozzi~\cite{Colding_1,Colding_Minicozzi,Colding_Minicozzi_2} for the level sets flow of the Green's functions in complete manifolds with nonnegative Ricci curvature. 
In fact, as explained in Subsection~\ref{sub:level}, the first ones can be recovered from ours in the limit as $p\to1^+$, whereas the latter can be reconstructed setting $p=2$ and letting $\Omega$ shrink to a single point (see the Appendix of~\cite{Ago_Fog_Maz}).

\medskip

We pass now to describe the main geometric inequalities obtained in this paper.
The first one is an extension of the Minkowski Inequality, holding for every bounded and smooth subset of $\R^n$, in which the total mean curvature of the boundary is replaced by the $L^1$-norm of the mean curvature, whereas the perimeter of the set $\Omega$ is replaced by the one of its {\em strictly outward minimising hull} $\Omega^*$, which is defined in Definition~\ref{smh-def} below in accordance to~\cite[pp. 371--372]{Hui_Ilm}. For the reader's convenience we briefly recall that a set is called {\em outward minimising} if it minimises the perimeter among all the sets containing it; moreover, an outward minimising set is called {\em strictly outward minimising} if it coincides almost everywhere with any outward minimising set containing it and having the same perimeter. 
Loosely speaking, $\Om^*$ is -- up to negligible components -- the smallest strictly outward minimising set that contains $\Omega$ (see Definitions~\ref{def:out} and~\ref{smh-def} in Section~\ref{sec:min} for more details). With these concepts at hand, our first main result reads as follows.

\begin{theorem}[Extended Minkowski Inequality]
\label{minkowski}
If $\Omega \subset \R^n$ is a bounded open set with smooth boundary, then
\begin{equation}
\label{minkowskif}
\left(\frac{|\pa \Omega^*|}{|\Sf^{n-1}|}\right)^{\!\!\frac{n-2}{n-1}} \!\!\! \leq \, \frac {1}{|\Sf^{n-1}|} \,\, \int\limits_{ \pa \Omega}   \left|\frac{\HH}{n-1}\right|  \, \dd \sigma \, ,
\end{equation}
where $\Om^*$ is the strictly outward minimising hull of $\Om$ defined as in Definition~\ref{smh-def}. Moreover, the dimensional constants appearing here
are optimal, in the sense that 
\begin{equation}
\min \left\{  \left. |\pa \Om^*|^{-\frac{n-2}{n-1}}\! \int\limits_{\pa\Om} \! |\HH|  \dd \sigma  \,\, \right| \,\, \Omega \Subset \R^n \, , \,\, \hbox{with $\pa \Om$ smooth}    \right\}  \, = \,\, (n-1) \, \abs{\Sf^{n-1}}^{\frac{1}{n-1}} \, ,
\end{equation}
and the minimum is achieved on spheres.
\end{theorem}

As a matter of fact, the Extended Minkowski Inequality~\eqref{minkowskif} is deduced as the limit, for $p\to 1^+$, of the following geometric $p$-capacitary inequality, which we believe of independent interest.
\begin{theorem}[$L^p$-Minkowski Inequality]
\label{premink}
Let $\Omega \subset \R^n$ be an open bounded set with smooth boundary. Then, for every $1<p<n$, the following inequality holds
\begin{equation}
\label{preminkf}
{\rm C}_p (\Omega)^{\frac{n-p-1}{n-p}}  \leq \,  \frac{1}{\,  \abs{\Sf^{n-1}} \,}\, \int\limits_{\partial \Omega} \, \left\vert \frac{\HH}{n-1}\right\vert^{p}\! \dd\sigma  \, ,
\end{equation}
where ${\rm C}_p(\Om)$ is the normalised $p$-capacity of $\Omega$ introduced in Definition~\ref{nor-pcap}.
Moreover, equality holds in~\eqref{preminkf} if and only if $\Omega$ is a ball.
\end{theorem}  
In order to deduce~\eqref{minkowskif} from~\eqref{preminkf}, one needs to compute the limit of the $p$-capacity of a bounded set with smooth boundary as $p \to 1^+$. 
Apart from the case of convex domains, 
treated in~\cite{Xiao},
we were unable to find  in the literature a complete and satisfactory discussion of this very basic issue. For this reason,  we have established that
\begin{equation}
\lim_{p\to 1^+} {\rm C}_p (\Omega) \, = \, \frac{\abs{\partial \Omega^*}}{ \abs{\Sf^{n-1}}}
\end{equation}
in Theorem~\ref{limit-pcapth} of Section~\ref{sec:min}.
%We acknowledge that relevant steps in the proof of the above inequality are inspired by the book of Maz'ya \cite{mazia}.

\medskip

As an immediate corollary of Theorem~\ref{minkowski}, we recover the Minkowski Inequality for {\em outward minimising} sets, since for every $\Om$ in this class it holds $|\pa \Om| = |\pa \Om^*|$ (see Remark~\ref{alt-def}) and $\HH \geq 0$, as a standard variational computation readily shows. Such inequality was originally conceived by Huisken in~\cite{Hui_video}, exploiting the theory of weak solutions to the IMCF, previously developed in~\cite{Hui_Ilm} (see also~\cite[Theorem 2--(b)]{freire} for a published version of the argument in the case of outward minimising sets with strictly mean-convex boundary). The conceptual differences between this method and our technique are described in Section~\ref{sec:method}, where we also discuss some of the technical subtleties arising in the two approaches.

\begin{corollary}[Minkowski Inequality for Outward Minimising Sets]
\label{meanconvexki}
If $\Omega \subset \R^n$ is a bounded outward minimising open 
set with smooth boundary, then
\begin{equation}
\label{meanconvexkif}
\left(\frac{|\pa \Omega|}{|\Sf^{n-1}|}\right)^{\!\!\frac{n-2}{n-1}} \!\!\! \leq \, \frac {1}{|\Sf^{n-1}|} \,\, \int\limits_{ \pa \Omega}   \frac{\HH}{n-1}  \, \dd \sigma \, .
%\left(\frac{|\Sf^{n-1}|}{|\pa \Omega|} \right)^{\!\!1/(n-1)} \!\!\!\!\! \leq \,\,\,\,\, \fint\limits_{\pa \Omega} \,  \frac{\HH}{n-1}  \, \dd \sigma \, .
%\abs{\partial \Omega}^{\frac{n-2}{n-1}} \abs{\Sf^{n-1}}^{\frac{1}{n-1}} \leq \int\limits_{\partial \Omega} \frac{\HH}{n-1}\,  \dd\sigma\, . 
\end{equation}
Moreover, the dimensional constants appearing here
are optimal, in the sense that 
\begin{equation}
\min \left\{  \left. |\pa \Om|^{-\frac{n-2}{n-1}}\! \int\limits_{\pa\Om} \! \HH \, \dd \sigma  \,\, \right| \,\, \Omega \Subset \R^n  \hbox{outward minimising}\, , \,\, \hbox{with $\pa \Om$ smooth}    \right\}  \, = \,\, (n-1) \, \abs{\Sf^{n-1}}^{\frac{1}{n-1}} \, ,
\end{equation}
and the minimum is achieved on spheres. Viceversa, if the equality holds in~\eqref{meanconvexkif} for some bounded {\em strictly outward minimising} open 
set with smooth and {\em strictly mean-convex} boundary, then $\Omega$ is isometric to a round ball.
\end{corollary}

It is worth pointing out that for $n \leq 7$, the above corollary can be shown to be equivalent to Theorem~\ref{minkowski} at the price of a non completely trivial approximation procedure. Indeed, starting from a possibly non outward minimising $\Omega$, one immediately checks that 
\begin{equation}
\frac {1}{|\Sf^{n-1}|} \,\, \int\limits_{ \pa \Omega^*}   \frac{\HH}{n-1}  \, \dd \sigma \,\, \leq
\,\,\frac {1}{|\Sf^{n-1}|} \,\, \int\limits_{ \pa \Omega}   \left|\frac{\HH}{n-1}\right|  \, \dd \sigma \, ,
\end{equation}
where $\Om^*$ is the strictly outward minimising hull of $\Om$. On the other hand, when $n \leq 7$ the main result in~\cite{sternberg-williams} guarantees that $\pa \Omega^*$ is of class $\mathscr{C}^{1,1}$. In particular, with the help of~\cite[Lemma 2.6]{huisken-ilm_higher}, it can be approximated in the $\mathscr{C}^{1} \cap W^{2,1}$-topology by a sequence of smooth and strictly mean-convex hypersurfaces. Arguing as in~\cite[Lemma 5.6]{Hui_Ilm} one realises that these hypersurfaces are bounding strictly outward minimising sets $\{ \Om^*_\ep \}_{\ep>0}$, so that inequality~\eqref{meanconvexkif} is satisfied by the whole approximating sequence. Letting $\ep \to 0$ yields
\begin{equation}
\left(\frac{|\pa \Omega^*|}{|\Sf^{n-1}|}\right)^{\!\!\frac{n-2}{n-1}} \!\!\! \leq \, \frac {1}{|\Sf^{n-1}|} \,\, \int\limits_{ \pa \Omega^*}   \frac{\HH}{n-1}  \, \dd \sigma \, ,
\end{equation}
from which the thesis follows at once. Obviously, this argument breaks down in higher dimension, due to the serious regularity issues coming from the possible presence of a non-empty singular set in the area minimising portion of $\pa\Om^*$.

\medskip

A simple and very nice consequence of inequality~\eqref{meanconvexkif} is a nearly umbilical estimate for {\em outward minimising} surfaces in $\R^3$ with {\em optimal constant}. The relation between the Minkowski Inequality and the nearly umbilical estimates was suggested by Huisken in~\cite{Hui_video}. Here, for the sake of reference, we included a proof of this fact in Section~\ref{sec:cor} (see Theorem \ref{delellis-muller}). The general nearly umbilical estimate for surfaces in $\R^2$  with an implicit dimensional  constant is a very remarkable theorem, proved in~\cite{DeLellis_Mueller} by De~Lellis and M\"uller.
We refer the reader to the original paper~\cite{DeLellis_Mueller} as well as to the Ph.D. thesis~\cite{perez-thesis} and the references therein for a complete account about the geometric features and implications of such a deep result.

\medskip

Another direct consequence of Theorem~\ref{minkowski} is the following optimal version of inequality~\eqref{eq:minkvol}, holding
for bounded open sets with smooth boundary.

%
%Applying the Isoperimetric Inequality to the left hand side of~\eqref{minkowskif}, and taking into account that
%%, being a  hull, 
%$\Omega\subseteq \Omega^*$, we deduce at once the volumetric version of the Minkowski Inequality, holding
%%Higher Order Isoperimetric Inequality 
%for bounded open sets with smooth boundary.
\begin{theorem}[Volumetric Minkowski Inequality]
\label{minkvol}
Let $\Omega \subset \R^n$ be a bounded open set with smooth boundary. Then
\begin{equation}
\label{minkvolf}
\left(\frac{| \Omega|}{|\mathbb{B}^{n}|}\right)^{\!\!\!\frac{n-2}{n}} \!\!\! \leq \, \frac {1}{|\Sf^{n-1}|} \,\, \int\limits_{ \pa \Omega}   \left|\frac{\HH}{n-1}\right|  \, \dd \sigma \, .
\end{equation}
Moreover, equality holds  in \eqref{minkvolf} if and only if $\Omega$ is a ball.
\end{theorem}
Inequality~\eqref{minkvolf} is achieved
applying the Isoperimetric Inequality to the left hand side of~\eqref{minkowskif}, and taking into account that $\Omega\subseteq \Omega^*$, the latter being a hull.
Observe that the rigidity statement in the above theorem follows essentially from the rigidity statement of the Isoperimetric Inequality. To the authors' knowledge, the above inequality was previously known to hold for domains with a strictly mean-convex boundary of positive scalar curvature (for short $\pa \Omega \in \Gamma_2^+$). On this regard, we refer the reader to the paper~\cite{Chang_Wang_2013} and the subsequent~\cite{qiu}, where the inequality is proved with methods based on Optimal Transport. For $n\leq 7$, we note that inequality~\eqref{minkvolf} can also be deduced from Corollary~\ref{meanconvexki} through the approximation argument outlined above.

%%%%%%%%%%%%%%%%%%%%%%%%%%%%%%%%%%%%%%%%%%%%%%%

\subsection{Summary}

%%%%%%%%%%%%%%%%%%%%%%%%%%%%%%%%%%%%%%%%%%%%%%

In Section~\ref{sec:method} we describe the main features of our method through a fairly systematic comparison with the previous approaches, based on the IMCF. Approximations schemes {\em \`a la} Moser~\cite{moser-jems} and formal analogies are employed to make some heuristic considerations as well as to introduce the main technical challenges of the present work. After collecting some preparatory material in Section~\ref{material}, we face these challenges in Section~\ref{proof}, which constitutes the core of our analysis. There we solve the issues coming from the presence of critical points, proving {\em effective monotonicity formulas} (see Theorem~\ref{monotone1} and Theorem~\ref{monotone2}), whose validity persists also beyond possible jumps. Having this tools at hand, we prove the $L^p$-Minkowski Inequality~\eqref{preminkf} and then, passing to the limit as $p\to 1^+$, we prove, in Section~\ref{sec:min}, the Extended Minkowski Inequality~\eqref{minkowskif}. Here the main difficulty is to characterise in a geometrically meaningful way the limit of the $p$-capacity of the domain under consideration. We accomplish this task in Theorem~\ref{limit-pcapth}. 

%Finally, in Section~\ref{sec:cor} we present some corollaries and related results. Among them we mention an optimal version of the well celebrated De Lellis-M\"uller Nearly Umbilical Estimates for outward minimising domains (Theorem~\ref{delellis-muller}) and the proof
%that any starshaped mean-convex domain is necessarily outward minimising (Theorem~\ref{th-star}).
%
%%%%%%%%%%%%%%%%%%%%%%%%%%%%%%%%%%%%%%%%%%%%%%%%%%
%%%%%%%%%%%%%%%%%%%%%%%%%%%%%%%%%%%%%%%%%%%%%%%%%%

\section{Inverse Mean Curvature Flow Versus Nonlinear Potential Theory} 
\label{sec:method}

%%%%%%%%%%%%%%%%%%%%%%%%%%%%%%%%%%%%%%%%%%%%%%%%%%
%%%%%%%%%%%%%%%%%%%%%%%%%%%%%%%%%%%%%%%%%%%%%%%%%%

Having introduced the main results of this paper in terms of geometric inequalities, we now describe the method that will be employed to deduce them. The most appropriate way to accomplish this task is to compare 
our approach set in nonlinear potential theory
with the one based on the IMCF, 
in both its smooth and weak version.

%%%%%%%%%%%%%%%%%%%%%%%%%%%%%%%%%%%%%%%%%%%%%%%%%%

\subsection{Smooth and Weak Inverse Mean Curvature Flow.}
\label{smuteuic}

%%%%%%%%%%%%%%%%%%%%%%%%%%%%%%%%%%%%%%%%%%%%%%%%%%

Using the smooth IMCF it is  possible to provide an extension of the classical Minkowski Inequality~\eqref{eq:mink2} for convex domains 
to the family of {\em starshaped domains} with {\em strictly mean-convex boundary}. This approach has been completely developed in~\cite{Guan_Li} and relies essentially on the results in~\cite{gerhardt} and~\cite{urbas}, where it is proven that if $ \Om$ is strictly mean-convex and starshaped, then the IMCF $\{\partial \Om_t \}_{t\geq 0}$ starting at $\pa \Omega$ is defined and smooth for all times. In this case it is possible to carry out a smooth computation, showing that the function 
\begin{equation}
\label{eq:mon}
t \,\, \longmapsto \,\, 
%\mathscr{F} (\Sigma_t) \, = \, 
|\pa\Om_t|^{-\frac{n-2}{n-1}} \!\int\limits_{\pa\Om_t} \!\HH \,\dd \sigma_{t}\, 
%\qquad \hbox{with} \,\,\,\, \Sigma_0 = \pa \Omega
\end{equation}
is non increasing. The Minkowski Inequality then follows 
from the main theorems in~\cite{gerhardt, urbas}. These ensure that, as $t \to + \infty$, the hypersurfaces $\pa \Om_t$ converge to a round sphere, once they are suitably rescaled, so that 
\begin{equation}
\label{eq:monlim}
|\pa \Om|^{-\frac{n-2}{n-1}} \! \int\limits_{\pa \Om} \!\HH \,\dd \sigma \,\, \geq \,\, \lim_{t \to + \infty} |\pa\Om_t |^{-\frac{n-2}{n-1}} \!\int\limits_{\pa\Om_t} \!\HH \,\dd \sigma_{t} \,\, = \,\, (n-1) \, |\Sf^{n-1}|^{\frac{1}{n-1}} \, .
\end{equation}
This approach, which is extremely clean and quite flexible, 
has found remarkable applications also in non Euclidean contexts (see~\cite{Brendle,ge-hyperbolic1,ge-hyperbolic2,ge-kottler,deLima-hyperbolic,Mak_She}). However, it is suitable only for those hypersurfaces that do not change topology along their evolution. 
For example, if the ambient manifold is (asymptotically) flat, it applies only to hypersurfaces  with spherical topology. 
%In fact, a starshaped hypersurface, being a graph over $\Sf^{n-1}$, is automatically diffeomorphic to $\Sf^{n-1}$. 

These topological restrictions can be overtaken considering weak solutions of the IMCF starting at the boundary of {\em outward minimising sets}, as described in Huisken-Ilmanen's theory~\cite{Hui_Ilm,huisken-ilm_higher}. In fact, weak solutions are engineered in order to allow for jumps,
at which the topological changes take place, 
preserving at the same time the monotonicity of the quantity~\eqref{eq:mon}. On this regard, it is worth pointing out explicitly that {\em the class of outward minimising sets includes the one of starshaped domains with strictly mean-convex boundary}, as it can be easily deduced combining the long time existence established in~\cite{gerhardt, urbas} with a simple integration by parts argument.

%Since we found the literature quite confusing on this point, we give a direct proof of this basic but fundamental fact in Theorem~\ref{th-star}. This follows somehow from a more general principle, stated in Theorem~\ref{prop-out2}, saying that if $\Omega$ is not outward minimising, then the smooth IMCF starting at $\partial \Omega$ cannot escape completely from the strictly outward minimising hull $\Omega^*$.
%It is also well known that for starshaped domains with strictly mean-convex boundary the smooth IMCF coincides with the weak IMCF for all times~\cite[Smooth Flow Lemma 2.3]{Hui_Ilm}, and thus their treatment turns out to be strictly included in the treatment of outward minimising sets, even under the point of view of the analytical method employed. 

Needless to say that the larger generality obtained through weak solutions comes at the cost of a much more sophisticated and delicate theory, whose extension to different contexts is not straightforward at all. For example, one of the major difficulties is to show that the monotonicity formulas survive the jumps. In~\cite{Hui_Ilm} this is achieved by means of an elliptic regularisation procedure in which the weak solution of the IMCF is approximated by a family of smooth functions whose level sets obey a slightly modified version of the desired monotonicity. To the best of our knowledge, such a spectacular though technically demanding construction has never been replied beyond the original context of asymptotically flat Riemannian manifolds~\cite{Hui_Ilm,Hui_video,freire,Wei,mccormick}, with the only exception of \cite[Theorem 3.2]{Lee_Nev}, where the authors have checked that Huisken-Ilmanen theory applies to the
case under consideration. Hence, the expected extensions of the results in~\cite{Brendle,ge-hyperbolic1,ge-hyperbolic2,ge-kottler,deLima-hyperbolic} to the case of outward minimising hypersufaces are missing so far.

Another crucial point in the weak IMCF theory is the characterisation of the limits of the relevant monotone quantities. To be more specific, one would like to extend the validity of~\eqref{eq:monlim} to the class of weak solutions of the IMCF starting at an outward minimising domain. A possible way to accomplish this task is to invoke~\cite[Theorem 2.7-(b)]{huisken-ilm_higher} and then again~\cite[Theorem 0.1]{gerhardt}. Indeed, the former affirms that the weak IMCF eventually becomes strictly starshaped at some large enough time, so that the latter applies yielding the desired asymptotic behaviour. The proof of~\cite[Theorem 2.7-(b)]{huisken-ilm_higher} relies in turn on a blow-down analysis that can be carried out following the proof of~\cite[Blowdown Lemma 7.1]{Hui_Ilm}.
{This argument represents a quite delicate point, to treat which a remarkable amount of theory needs to be employed.} {In the present paper, as explained in the next subsection, both the problem of dealing with possible jumps and the problem of characterising the limit of the relevant monotone quantities will be faced in a completely different way, with the help of nonlinear potential theory.}

%%%%%%%%%%%%%%%%%%%%%%%%%%%%%%%%%%%%%%%%%%%%%%%%%

\subsection{Level sets of $p$-capacitary potentials.}
\label{sub:level} 

%%%%%%%%%%%%%%%%%%%%%%%%%%%%%%%%%%%%%%%%%%%%%%%%%

The key point in our approach is to replace the
weak IMCF technique with a novel analysis of a very natural family of functions, namely the $p$-capacitary potentials of $\Omega$. 
These are the weak solutions to the problems 
\begin{align}
%\label{pb}
\left\{\begin{array}{lll}
\Delta_p u=0 & \mbox{in} & \R^n\setminus \overline{\Omega} \, ,\\
\quad \,\,u=1 & \mbox{on} & \partial\Omega \, ,\\
\, u(x)\to 0 & \mbox{as} & |x|\to\infty \, ,
\end{array}\right.
\end{align}
with $1<p<n$.

 For heuristic considerations, it is useful to recall that these functions are playing an ancillary role in the weak IMCF theory, by virtue of a beautiful approximation result, originally due to R.~Moser~\cite{moser-jems} and subsequently extended by Kotschwar and Ni~\cite{kotschwar} (see also the very recent~\cite{mari-rigoli-setti}). This result says that if $u_p$ is a weak solution to the above problem, then the functions $w_p = -(p-1) \log u_p$ converge locally uniformly in $\R^n\setminus {\Omega}$ to a proper weak solution of the IMCF, as $p\to 1^+$. As a formal justification of this fact,
it can be noted that the functions $w_p$ satisfy the identities
\begin{equation*}
\Delta_p w_p \, = \,\,  \abs{\D w_p}^p ,
\end{equation*}
which in the limit as $p \to 1^+$ becomes  
\begin{equation}
\label{eq:levelsets}
\dive\left(\frac{\D w}{\abs{\D w}}\right)  = \,\,\abs{\D w} \, . \phantom{\qquad \quad}
\end{equation}
As it is well known, the latter equation corresponds to the level set formulation of the IMCF, since the left hand side equals the mean curvature of the level sets of $w$. 

As a concrete application of this elegant approximation scheme one immediately earns an alternative and flexible approach to the {\em existence} theory for the weak IMCF. Unfortunately, such a scheme is not as much effective for drawing geometric conclusions. For example, one would hopefully like to use it to deduce the validity of~\eqref{eq:monlim} from the nowadays classical asymptotic expansions of the $u_p$'s (see Lemma~\ref{asyu}). However, the convergence is not strong enough to take the limit of the integral quantities, moreover it only holds on compact subsets of $\R^n \setminus \Om$, preventing this strategy from being successfully implemented.

In view of these considerations and motivated by the recent discoveries in the field of harmonic functions~\cite{Ago_Maz_1,Ago_Maz_3,Ago_Fog_Maz}, we then let the $p$-harmonic functions play the leading role, providing their already rich and well established theory with a so far missing ingredient, namely the existence of {\em monotonicity formulas} -- or a relaxed version of them -- holding along their level set flow and in presence of critical points. Once this will be done, we would get a self consistent theory, completely independent of the IMCF and powerful enough to recover and extend all the main geometric conclusions.

To clarify these concepts, let us first discuss the toy-problem case, where the $p$-capacitary potential has no critical points. 
In this case, which is treated in~\cite{Fog_Maz_Pin} 
assuming $\Omega$ convex, 
one finds that for every $1<p<n$ the function
\begin{equation}
\label{eq:p-mon}
(0,1]  \ni \, \tau \, \longmapsto \,\, 
U_p(\tau) \, = \, \tau^{-\frac{n-1}{n-p}} \!\!\!\int\limits_{\{ u_p = \tau \}} \!\!\! |\D u_p |^p \,\dd \sigma\, 
\end{equation}
is nondecreasing. The monotonicity readily implies~\eqref{preminkf}. In fact, computing the limit of $U_p$ as $\tau \to 0^+$, with the help of Lemma~\ref{asyu} (see~\cite[Lemma 2.6]{Fog_Maz_Pin} with $q=p/(p-1)$ for details), gives
\begin{equation}
\label{eq:limiticchio}
\Big(\frac{n-p}{p-1}\Big)^{\!p} \,  \abs{\Sf^{n-1}}\,\, {\rm C}_p (\Omega)^{\frac{n-p-1}{n-p}}  \, = \,  \lim_{\tau \to 0^+} U_p(\tau) \,\, \leq \,\, U_p(1) \, = \, \int\limits_{\partial \Omega} \abs{\D u_p}^{p} \dd\sigma \, .
\end{equation}
Next, one computes 
\begin{equation}
\label{eq:derivaticchia}
0 \,\,\leq \,\, U_p'(1) \, = \, \frac1{(p-1)} \int\limits_{\partial \Omega} \abs{\D u_p}^{p-1} \HH \dd\sigma  	\,\, -   \,\, \Big(\frac{n-1}{n-p}\Big) \int\limits_{\partial \Omega} \abs{\D u_p}^{p} \dd\sigma \, , 
\end{equation}
and thus, by the H\"older inequality, one gets
\begin{equation*}
\int\limits_{\partial \Omega} \abs{\D u_p}^{p} \dd\sigma \,\, \leq \,\,\Big(\frac{n-p}{p-1}\Big)^{\!p} \int\limits_{\partial \Omega} \left\vert \frac{\HH}{n-1}\right\vert^{p}\! \dd\sigma \, .
\end{equation*}
Combining the latter inequality with~\eqref{eq:limiticchio} 
yields~\eqref{preminkf}. 
The Extended Minkowski Inequality~\eqref{minkowskif} can thus be obtained in the limit of 
\eqref{preminkf}, as $p \to 1^+$, 
by using the analysis of Section~\ref{sec:min}.

%\bigskip
%\bigskip
%
% {\color{red} DIRE QUI O DOPO COME SI DEDUCE LA P-MINKOWSKI. FORSE E' MEGLIO DIRLO QUI E POI}
% 
%

A question arises whether the monotonicity of the $U_p$'s and the one of the quantity defined in~\eqref{eq:mon} are related or not. It turns out that the only way to answer this question is to proceed formally.
%In order to understand the relation between the monotonicity of $U_p$ and the one of the quantity defined in~\eqref{eq:mon}, 
%it is convenient to proceed formally. 
Setting as above $w_p = -(p-1) \log u_p$ and $t = -(p-1) \log \tau$, for every $1<p<n$, the monotonicity of the function 
$U_p$ is equivalent to the statement that the function
\begin{equation}
%\label{eq:p-mon2}
[0,+\infty)  \ni \, t \, \longmapsto \,\, 
{\rm e}^{-\frac{n-p-1}{n-p} \, t} \!\!\!\!\int\limits_{\{ w_p = t \}} \!\!\!\! |\D w_p |^p \,\dd \sigma\,
\end{equation}
is nonincreasing. Taking the formal limit as $p\to 1^+$, one would get the same monotonicity statement for the function
\begin{equation}
%\label{eq:p-mon3}
[0,+\infty)  \ni \, t \, \longmapsto \,\, 
{\rm e}^{-\frac{n-2}{n-1} \, t} \!\!\!\!\int\limits_{\{ w = t \}} \!\!\!\! |\D w | \,\dd \sigma\, \, ,
\end{equation}
where $w$ solves~\eqref{eq:levelsets}, and thus $|\D w| (x)$ coincides with the mean curvature of the level set passing through $x$. Recalling that $|\{w=t\}| = |\pa \Om_t| = |\pa \Om| \,  {\rm e}^t$ along the IMCF, it is easy to realise that the latter monotonicity is equivalent to the one in~\eqref{eq:mon}. However, it is important to emphasize once more that the above argument is just formal and does not prove anything, since $w_p$ is converging to $w$ only locally uniformly on compact sets and $w$ itself is nothing more than a weak solution to the IMCF.

%
%However, this suggests that monotonicity properties of the functions~\eqref{eq:p-mon} are related to the monotonicity of the function~\eqref{eq:mon}.
%

In presence of critical points for the $p$-capacitary potentials, the above formal derivation could appear in principle all the more \emph{na\"ive}, since -- unlike in the linear case
treated in~\cite{Ago_Maz_3} -- 
the monotonicity of~\eqref{eq:p-mon} is not even
\emph{a priori} guaranteed. This phenomenon is typical of the nonlinear setting and is basically due to the loss of analyticity of the solution and the consequent loss of control on the behaviour of the critical points and of the critical values. 
In particular, in the case of the $p$-capacitary potential, 
one cannot exclude {\em a priori} the presence of clusters of critical points and critical values with full measure. Due to these difficulties, it is impossible to re-adapt the strategy employed in the linear case~\cite{Ago_Maz_3} to earn the {\em full monotonicity} of the $U_p$'s. However, we will be able to prove in the next sections that the inequalities 
\begin{equation}
\label{eq:crucial}
0 \,\, \leq \,\, U_p'(1)
 \qquad \hbox{and} \qquad  \lim_{\tau \to 0^+} U_p(\tau) \,\, \leq \,\, U_p(1) \, 
\end{equation}
hold true. These inequalities -- together with their conformal counterparts defined below~\eqref{crucial_phi} -- will be referred to as {\em effective inequalities} and will be deduced in Section~\ref{proof} as consequences of some {\em effective monotonicity}
properties of the functions $U_p$ (see Theorems~\ref{monotone1} and~\ref{monotone2} combined with formul\ae~\eqref{upfip}). 
Here, the locution ``effective monotonicity'' should be understood in contrast with the ``full monotonicity'' which is instead enjoyed by the $U_p$'s either in the case where
$\Omega$ is convex 
(see~\cite{Fog_Maz_Pin})
or in the case $p=2$
(see~\cite{Ago_Maz_3}).

As explained through~\eqref{eq:limiticchio} 
and~\eqref{eq:derivaticchia}, the two conditions
in~\eqref{eq:crucial} are sufficient to deduce the $L^p$-Minkowski Inequality~\eqref{preminkf} and in turn the Extended Minkowski Inequality~\eqref{minkowskif}. It must be noticed that the proof of the inequalities~\eqref{eq:crucial} requires both a technical and a conceptual enhancement of the previously existing techniques (\cite{Ago_Maz_3, Ago_Maz_2, Fog_Maz_Pin}). This is particularly evident in the analysis leading to the second inequality, which is based on the discovery of a further family of monotonic functions
(see Theorem~\ref{monotone2}), whose existence was far beyond the horizon both of the linear case~\cite{Ago_Maz_3} 
and of the nonlinear convex case~\cite{Fog_Maz_Pin}.

%%%%%%%%%%%%%%%%%%%%%%%%%%%%%%%%%%%%%%%%%%%%%%%%%%

\subsection{Further directions.}

%%%%%%%%%%%%%%%%%%%%%%%%%%%%%%%%%%%%%%%%%%%%%%%%%%

In virtue of the previous observations it is quite clear that methods based on linear and nonlinear potential theory may provide an efficient alternative to the employment of the IMCF techniques in many contexts. To be concrete, let us just mention a couple of projects that represent a natural continuation of the present work. 

The first one is the potential revisitation of the IMCF proof of the Riemannian Penrose Inequality due to Huisken and Ilmanen~\cite{Hui_Ilm}. 
%\begin{comm}(in the spirit of the potential theoretic studies 
%of the positive mass theorem~\cite[CHRUSCIEL, 
%KIOWSKY, JEREWSKY]{})%\end{comm}. 
Indeed, it is not too hard to provide a formal guess of the 
monotonic quantities which are expected to play the same role as the Hawking Mass in the Geroch's monotonicity scheme. The hard part, as usual, is the treatment of the critical points. There are good reasons to believe that the method and the ideas presented in this work also apply to that situation. 

The second one is the extension of the Minkowski Inequalities to the case of complete manifolds with nonnegative Ricci curvature. Proceeding in parallel with the linear theory (compare~\cite[Theorem $1.1$]{Ago_Maz_3} with~\cite[Theorem $1.3$]{Ago_Fog_Maz}), 
one can prove that if the manifold $(M,g)$ is Asymptotically Locally Euclidean and $\Om$ is an outward minimising domain, then it holds
\begin{equation*}
\left({\rm AVR}(g) \, \frac{|\Sf^{n-1}|}{|\pa \Omega|} \right)^{\!\!1/(n-1)} \!\!\!\!\! \leq \,\,\,\,\, \fint\limits_{\pa \Omega}\frac{\HH}{n-1}  \,\, \dd \sigma \, ,
\end{equation*}
where ${\rm AVR}(g)$ stands for the Asymptotic Volume Ratio of $(M,g)$. A detailed proof of this result will appear in a forthcoming manuscript. 

Other challenges include the study of natural geometric inequalities in Cartan-Hadamard manifolds as well as in Asymptotically Hyperbolic manifolds.

%%%%%%%%%%%%%%%%%%%%%%%%%%%%%%%%%%%%%%%%%%%%%%%%%%
%%%%%%%%%%%%%%%%%%%%%%%%%%%%%%%%%%%%%%%%%%%%%%%%%%

\section{Preparatory material}
\label{material}

%%%%%%%%%%%%%%%%%%%%%%%%%%%%%%%%%%%%%%%%%%%%%%%%%%
%%%%%%%%%%%%%%%%%%%%%%%%%%%%%%%%%%%%%%%%%%%%%%%%%%

%%%%%%%%%%%%%%%%%%%%%%%%%%%%%%%%%%%%%%%%%%%%%%%%%%

\subsection{Preliminaries on $p$-capacitary potentials}

%%%%%%%%%%%%%%%%%%%%%%%%%%%%%%%%%%%%%%%%%%%%%%%%%%

We recall the well known notion of $p$-capacity, introducing at the same time a normalised version of it that is suitable for our applications. 

%\begin{comm}use\end{comm} 
%the following definition, that coincides with the classical one,
%\begin{comm}up to the multiplicative constant
%$[(p-1)/(n-p)]^{p-1}$.
%\end{comm}
 
\begin{definition}[$p$-capacity \& normalised $p$-capacity]
\label{nor-pcap}
Let $\Omega\subset\R^n$ be a bounded open set with smooth boundary. 
\begin{itemize}
\item The $p$-capacity of $\Omega$ is defined as 
\begin{equation}
\label{pcap}
\capa_p(\Omega) \, = \, \inf\left\{ \int_{\mathbb{R}^n} \!|\D v|^p\, 
{\rm d}\mu 
\ \bigg\vert \ v\in \mathscr{C}^{\infty}_c(\mathbb{R}^n),\ v\geq 1\ \mbox{on}\ \Omega\right\}.
%\inf\left\{ \int_{\R^n} \abs{ \D f}^p \dd\mu , \quad f \geq \chi_\Omega, \,\, f \in C^\infty_0(\R^n) \right\}.
\end{equation}
\item The normalised $p$-capacity of $\Omega$ is defined as 
\begin{equation}
\label{pcapf}
{\rm C}_p(\Omega) \, = \, \inf\left\{\Big(\frac{p-1}{n-p}\Big)^{p-1}\!\!\frac{1}{\abs{\Sf^{n-1}}}\int_{\mathbb{R}^n} \!|\D v|^p\, 
{\rm d}\mu 
\ \bigg\vert \ v\in \mathscr{C}^{\infty}_c(\mathbb{R}^n),\ v\geq 1\ \mbox{on}\ \Omega\right\}.
\end{equation}
\end{itemize}
\end{definition}
The variational structure of the above definition leads naturally to the formulation of the following problem
\begin{align}
\label{pb}
\left\{\begin{array}{lll}
\Delta_p u=0 & \mbox{in} & \R^n\setminus \overline{\Omega} \, ,\\
\quad \,\,u=1 & \mbox{on} & \partial\Omega \, ,\\
\, u(x)\to 0 & \mbox{as} & |x|\to\infty \, .
\end{array}\right.
\end{align}
It is well known that, for every bounded open set $\Omega$ with smooth boundary and
 every $1<p<n$, problem~\eqref{pb} admits a 
unique weak solution. Such a solution is called the $p$-capacitary potential associated with $\Om$. For the reader's convenience, we recall that a function $v$ is 
a weak solution of $\Delta_p v = 0$ 
in an open set $V$ if $v \in W^{1, p}_{loc} (V)$ and 
\[
\int_V \left \langle \abs{\D v}^{p-2} \D v
\,\Big|\,\D \psi\right\rangle \dd\mu = 0
\] 
%\begin{comm}
%(USIAMO SEMPRE $\langle\cdot|\cdot\rangle$
%INVECE CHE $\langle\cdot,\cdot\rangle$)
%\end{comm}
for any test function 
$\psi \in \mathscr{C}^{\infty}_c (V)$.
%{\begin{comm} MEGLIO DIRE DIRETTAMENTE LA DEFINIZIONE DI SOLUZIONE DEBOLE DEL PROBLEMA~\eqref{pb}\end{comm}}. 
By the important 
contributions~\cite{dibenedetto-p-regularity, evans-p-regularity,lewis-reg} and~\cite{uraltseva},
%\cite{dibenedetto-p-regularity}, \cite{evans-p-regularity}, \cite{lewis-reg}, \cite{uraltseva}, 
we know that weakly $p$-harmonic functions are $\mathscr{C}^{1, \alpha}_{loc}$  
(we are not aware of an explicit formula
relating $\alpha$ to $p$; we note, however, that
such relation cannot be uniform in $p$).
On the other hand, the classical regularity theory for quasilinear nondegenerate elliptic equations (see e.g.~\cite{lady}) ensures that they are analytic  around the points where the gradient does not vanish.
We also recall from \cite{Lie} 
that the $\mathscr{C}^{1, \alpha}$-regularity
can be extended up to the boundary.

Note that the uniqueness of the solution to problem~\eqref{pb}
can be easily proved by suitably applying the
Comparison Theorem for weakly $p$-harmonic functions~\cite[Theorem 2.15]{lindqvist} on large balls of radius $R$, and letting then
$R\to+\infty$.
With the same argument one can also show that
the solution $u$ to problem~\eqref{pb} is such that
$0<u(x)<1$ for every $x\in\R^n\setminus\Omega$.
Finally, we recall that such a solution
realises the infimum in~\eqref{pcapf}.
This can be proved using a standard exhaustion scheme 
(for example the one proposed in~\cite{colesanti-salani})
and invoking the $\mathscr{C}^{1, \alpha}_{loc}$ regularity to guarantee the convergence of the scheme itself.
These facts are summarised in the following theorem.

\begin{theorem}[Existence and regularity of $p$-capacitary potentials]
Let $\Omega\subset\R^n$ be a bounded open set with smooth boundary,
and let $1<p<n$. Then, the following statements hold true:
\begin{itemize}
\item[(i)] There exists a unique weak solution 
$u\in \mathscr{C}^{1, \alpha}_{loc}(\R^{n} \setminus \overline{\Omega})
\cap \mathscr{C}(\R^{n} \setminus\Omega)$ to problem \eqref{pb}.
%(dovrebbe essere pi\`u giusto cos\`i tenendo conto
%della condizione al bordo e del fatto che
%$W^{1, p}_{loc}(\R^{n} \setminus \overline{\Omega})\cap C^{1, \alpha}_{loc}(\R^{n} \setminus \overline{\Omega})
%=C^{1, \alpha}_{loc}(\R^{n} \setminus \overline{\Omega})$, no?) 
\item[(ii)] The solution
$u$ is analytic on the complement 
%of $\crit(u)$, i.e.
at points where $\D u \neq 0$.
\item[(iii)] The solution $u$ fullfills
\begin{equation}
\label{capattained}
{\rm C}_p(\Omega)
\,=\,
\Big(\frac{p-1}{n-p}\Big)^{p-1} \!\frac1{\abs{\Sf^{n-1}}}
\int\limits_{\mathbb{R}^n\setminus\overline{\Omega}} \!\!|\D u|^p\, 
{\rm d}\mu \,,
\end{equation}
where ${\rm C}_p(\Om)$ is the normalised $p$-capacity of $\Om$ defined in~\eqref{pcapf}.
\end{itemize}
\end{theorem} 
Note that since $\partial \Omega$ is assumed to be smooth,
by the Hopf Lemma for $p$-harmonic functions
(see~\cite[Proposition 3.2.1]{tolksdorf}), we have that
$\abs{\D u} \neq 0$ in a neighbourhood of this hypersurface.
In particular, $u$ is analytic in such a neighbourhood. 
Coupled with this fact, the asymptotic expansions below imply that 
$\crit(u)=\big\{x\in\R^n\setminus\overline\Om\,\,\big| \, \,\D u(x)=0\big\}$ 
is a compact subset of $\R^n \setminus \overline{\Omega}$ (generically depending on $p$), and 
in turn that $u$ is analytic outside this set. Finally, it is worth recalling that for $p \neq 2$, the set $\crit(u)$ is {\em a priori} allowed to have full measure.

\begin{lemma}[Asymptotic expansions of $u$ and $\abs{\D u}$]
\label{asyu} 
Let $\Omega\subset\R^n$ be a bounded open set with smooth boundary, and let $1<p<n$. Then, the solution $u$ to~\eqref{pb} satisfies
\begin{itemize}
	\item[(i)] $\lim_{|x|\to+\infty} u(x) \, |x|^{\frac{n-p}{p-1}} \, = \, {\rm C}_p(\Omega)^{\frac{1}{p-1}}$ \!\! ,
	\item[(ii)] $\lim_{|x|\to+\infty} |\D u(x)| \, |x|^{\frac{n-1}{p-1}} \, = \, \big(\frac{n-p}{p-1}\big) \,  {\rm C}_p(\Omega)^{\frac{1}{p-1}}$ \!\! ,
\end{itemize}
where ${\rm C}_p(\Om)$ is the normalised $p$-capacity of $\Om$ defined in~\eqref{pcapf}. 
In particular, $\crit (u)$ is a compact 
subset of $\R^n \setminus \overline{\Omega}$, possibly with full measure.
\end{lemma}

For the proof of this lemma we refer the reader to~\cite{kikhenassamy} (see also the more recent~
\cite[Lemma 2.3 and (2.2)]{poggesi} for a precise statement). It is also worth mentioning~\cite{Ga_Sa}, where 
similar expansions are employed to infer rotational symmetry of starshaped domains supporting a solution to problem~\eqref{pb} with constant normal derivative on the boundary. From the point of view of the present paper, the main implication of the above lemma is the computation of the limit
\begin{equation}
\label{eq:limup}
\lim_{\tau \to 0^+} U_p(\tau) \,\, = \,\,\Big(\frac{n-p}{p-1}\Big)^{\!p}\,  \abs{\Sf^{n-1}}\,\, {\rm C}_p (\Omega)^{\frac{n-p-1}{n-p}} \, ,\end{equation}
where $\tau \mapsto U_p (\tau)$ is the function defined in~\eqref{eq:p-mon} (see~\cite[Lemma 2.6]{Fog_Maz_Pin}). The following characterization of the $p$-capacity of $\Omega$ is widely used in the literature and it is also very useful for our purposes. Hence, we provide a proof.

\begin{lemma}
\label{id-cap}
Let $\Omega\subset\R^n$ be a bounded open set with smooth boundary, 
and let $1<p<n$. Then, the solution $u$ to~\eqref{pb} satisfies
\begin{equation}
\label{cap-rel}
{\rm C}_p(\Omega)
\,\,=\, \, 
\Big(\frac{p-1}{n-p}\Big)^{p-1}\!\frac1{\abs{\Sf^{n-1}}}
\int\limits_{\partial \Omega} \, \abs{\D u}^{p-1} \dd\sigma \, ,
\end{equation}
where ${\rm C}_p(\Om)$ is the normalised $p$-capacity of $\Om$ defined in~\eqref{pcapf}.
\end{lemma}

\begin{proof}
For $\epsilon > 0$, let $V_\epsilon$ be 
the $\ep$-tubular neighbourhood of $\crit(u)$, namely
\[
V_\ep
\,=\,
\Big\{x\in\R^n\setminus\overline\Omega \,\, \Big| \,\, {\rm dist}
\big(x,\crit (u)\big)<\ep
\Big\} \, ,
\]
where ${\rm dist}\big(x,\crit (u)\big)$ is the Euclidean distance of 
$x$ from $\crit (u)$.
%(ho scritto $V_\ep$ esplicitamente e con una notazione
%diversa da $U_\ep$ per distinguerlo dai pigiami. Potremmo 
%ad un certo punto rimarcare come questi intorni
%pi\`u standard, a differenza dei nostri pigiami,
%non funzionerebbero per la dim. dei nostri risultati.)
By the compactness of $\crit(u)$
in $\R^n\setminus\overline\Omega$,
we have that $V_\epsilon \subset
\{u\geq t\}$,
for every $\ep>0$ and $t>0$ small enough.
Since $\abs{\D u} = 0$ on $\crit (u)$ by definition, 
we have from identity \eqref{capattained} and by 
Monotone Convergence Theorem that
\begin{equation*}
\Big(\frac{n-p}{p-1}\Big)^{\!p-1}\!\abs{\Sf^{n-1}} \, {\rm C}_p (\Omega) 
%\,= \int\limits_{\R^n \setminus\overline\Omega} \abs{\D u}^{p} \dd\mu 
\,\,= \, \lim_{\epsilon \to 0^+}\lim_{t\to 0^+} \!\!\!  \int\limits_{
\{u \geq t\}\setminus V_\epsilon
}\!\!\!\abs{\D u}^{p} \dd\mu \, .
\end{equation*}
By the discussions above, 
$u$ is analytic --  
and in turn $p$-harmonic in the classical sense --  
in the set 
$\{u \geq t\}\setminus V_\epsilon$.
Therefore, for $\ep$ and $t$ small enough,
the Divergence Theorem yields
\begin{align*}
\int\limits_{\{u \geq t\}\setminus V_\epsilon} 
\!\!\abs{\D u}^{p} \dd\mu 
&\,=\!\!\! 
\int\limits_{\{u \geq t\}\setminus V_\epsilon} 
\!\!{\rm div}\left(u \, |\D u|^{p-2}\D u\right)\dd\mu \, = \,
\\
&\,=\, 
-\,t\!\!\!\int\limits_{\{u=t\}}\!\!\!\abs{\D u}^{p-1} \dd\sigma 
\,+\, 
\int\limits_{\partial \Omega} \abs{\D u}^{p-1} 
\dd\sigma 
\,+\, 
\int\limits_{\partial V_\epsilon}\!\!
u\,\abs{\D u}^{p-2}
\left\langle \D u\,|\,\nu\right\rangle\dd\sigma \, ,
\end{align*}
where $\nu$ is the
inward unit normal to $V_\epsilon$. 
Observe that 
$\nu$ is well defined 
almost everywhere on 
$\partial V_\epsilon$ and for almost every 
$\epsilon>0$, 
in view of the Sard-type property for Lipschitz functions 
proved in \cite{Alb_Bia_Cri}. 
Letting $t \to 0^+$ the integral on $\{u=t\}$ tends to $0$ by the asymptotic expansion
(ii) of Lemma \ref{asyu}, while letting $\epsilon \to 0^+$ the integral on $\partial V_\epsilon$ tends to $0$ since $\abs{\D u}^{p-1} (x)$ vanishes as $x$ approaches $\crit(u)$
and since $0\leq u\leq 1$.
\end{proof}

\bigskip

\noindent In the following subsection as well as in the remaining part of the paper we will always assume that $1<p<n$, unless otherwise stated.

\bigskip

%%%%%%%%%%%%%%%%%%%%%%%%%%%%%%%%%%%%%%%%%%%%%%%%%%

\subsection{The conformal setting}
\label{conformalsetting}

%%%%%%%%%%%%%%%%%%%%%%%%%%%%%%%%%%%%%%%%%%%%%%%%%

As shown in \cite{Ago_Maz_3, Ago_Maz_2, Ago_Fog_Maz} and \cite{Fog_Maz_Pin}, 
it is very convenient to work in the conformally related Riemannian manifold $(\R^n \setminus \Omega, g)$, where $g$ is given by 
\begin{equation}
\label{g}
g \, = \, u^{2\big(\tfrac{p-1}{n-p}\big)}g_{\R^n}\, .
\end{equation}
In this setting it is also convenient to consider the new variable
\begin{equation}
\label{phi}
\phi\, = \,  - \, \frac{(p-1)(n-2)}{(n-p)} \,  \log u \, .
\end{equation}
By the same formal computations as in~\cite{Fog_Maz_Pin}, 
the boundary value problem~\eqref{pb} translates in terms of $g$ and $\phi$ as 
\begin{equation}
\label{prob_ex_rif}
\left\{
\begin{array}{rcll}
\displaystyle
\phantom{\frac12}\Delta_{p}^g \phi \!\!\!\!&=&\!\!\!\!0 & {\rm in }\quad
\R^n\setminus\overline\Om,\\
\displaystyle
\ric_g - \nabla\nabla\phi 
+\frac{d\phi\otimes d\phi}{n-2}
\!\!\!\!&=&\displaystyle\!\!\!\!
\left(\frac{|\nabla\phi|^2_g}{n-2}
-\Big(\frac{p-2}{n-2}\Big)
\frac{\nabla\nabla\phi(\nabla\phi,\nabla\phi)}{|\nabla\phi|_g^2}
\right)g & {\rm in }\quad
\big(\R^n\setminus\overline\Om\big){\setminus}\crit(\phi),\\
\displaystyle
 \phantom{\frac12}\phi\!\!\!\!&=&\!\!\!\!0  &{\rm on }\ \ \pa\Om,\\
\displaystyle
\phantom{\frac12}\phi(x)\!\!\!\!&\to&\!\!\!\!+\infty & \mbox{as }\ x\to\infty.
\end{array}
\right.
\end{equation}
Here, $\nabla$ is the Levi-Civita connection of $g$, $\nabla\nabla$ the Hessian operator, and $\Delta_p^g$ is the $p$-Laplace operator computed with respect to the metric $g$, explicitly defined as
\[
\Delta_p^g \phi \,  = \, \dive_{\!g}\big(\abs{\nabla \phi}^{p-2} \nabla \phi\big) \, ,
\]  
where $\dive_{\!g}$ is the divergence computed with respect to $g$. A very useful tool in the study of $p$-harmonic functions is the Kato-type identity, introduced in~\cite[Proposition 4.4]{Fog_Maz_Pin}.
For the reader's convenience, we recall its precise statement
in the following proposition.

\begin{proposition}[Kato-type identity \& orthogonal decomposition]
\label{kato}
 Let $(M, g)$ be a Riemannian manifold, and let $\phi$ be a $p$-harmonic function on $M$. Then, at any point where $\abs{\nabla \phi} \neq 0$, the following identity holds true
\begin{equation}
\label{katof}
\begin{split}
\hspace{-0.3cm}\abs{\nabla\nabla \phi}^2 - \bigg(1 + \frac{(p-1)^2}{n-1}\bigg) \bigabs{\nabla\abs{\nabla \phi}}^2 \!\! = &\,\,\abs{\nabla \phi}^2 \, \bigabs{ \, \hh - \frac{\HH}{n-1}\,g^{\!\top} \, }^2 
%\\  
%& 
\!\!+\bigg(1 - \frac{(p-1)^2}{n-1}\bigg) \bigabs{\nabla^{\top}\abs{\nabla \phi}}^2, 
\end{split}
\end{equation}
where $\hh$ and $\HH$ are respectively the second fundamental form and the mean curvature of the level sets of $\phi$ with respect to the unit normal $\nabla \phi /\abs{\nabla \phi}$, and $g^{\!\top}$ is the metric induced by $g$ on the level sets of $\phi$. Finally, for a given differentiable function $f$, we agree that $\nabla^\top \! f$ indicates the tangential part of the gradient, according to the orthogonal decomposition
\begin{equation}
\nabla^\perp f \, = \, \bigg\langle \nabla f \, 
\bigg| \, \frac{\nabla \phi}{|\nabla\phi|} \bigg\rangle
\frac{\na\phi}{|\na\phi|} 
\qquad \hbox{and} \qquad 
\nabla^\top \!f \, = \, \nabla f - \nabla^\perp f \,.
\end{equation}
In particular, the following formula holds true
\begin{equation}
\label{decomposition-intro}
\big\vert \nabla\abs{\nabla \phi} \big\vert^2   =  \,\,\big\vert \nabla^{\top} \abs{\nabla \phi} \big\vert^2 + \,  \big\vert \nabla^{\perp} \abs{\nabla \phi} \big\vert^2  .
\end{equation}
\end{proposition}

Since the proof of the $L^p$-Minkowski Inequality outlined in Subsection~\ref{sub:level} will be carried out in the conformal setting described above, the fundamental conditions~\eqref{eq:crucial} need to be rephrased accordingly. It is then worth introducing the following definition.
\begin{definition}[The function $\Phi_p$]
\label{phipdef}
For any $1< p <n$, let $g$ and $\phi$ be the solutions to~\eqref{prob_ex_rif} obtained by the solution to~\eqref{pb} through~\eqref{g} and~\eqref{phi}. We define the function $\Phi_p: [0, +\infty)  \to \R$ by
\begin{equation}
\label{phip}
\Phi_p(s) 
\,\,= \!\!\!
\int\limits_{\{\phi = s\} }
\!\!\abs{\nabla \phi}_g^{p} \, \dd\sigma_{\!g} \, ,
\end{equation}
where $\dd\sigma_{\!g}$ is the area element induced by the ambient measure $\dd \mu_{g}$ on the given level set. We agree that 
\begin{equation*}
\int\limits_{\{ \phi = s \} } \!\!\! |\nabla \phi |_g^p \,\dd \sigma_{\!g} \,\,\,\,= \!\!\!\!\!\!\!\!\!\!
\int\limits_{\{ \phi = s \} \setminus \crit(\phi)} \!\!\!\!\!\!\!\!\!\! |\nabla \phi |_g^p \,\dd \sigma_{\!g}\, ,
\end{equation*}
whenever a critical value is involved.

\end{definition}

We conclude this subsection, recalling some of the relevant properties of the function $\Phi_p$ just introduced. Their proofs are basically immediate -- as they follows from the analogous properties of the corresponding function $U_p$, defined in~\eqref{eq:p-mon} -- and are left to the reader.
\begin{itemize}
\item The function $\Phi_p$ is bounded at infinity. Moreover, it follows from~\eqref{eq:limup} that 
\begin{equation}
\label{limphi}
\lim_{s \to +\infty} \Phi_p (s) \,\, = \,\,  \Big(\frac{n-p}{p-1}\Big)^{\!p}\,  \abs{\Sf^{n-1}}\,\, {\rm C}_p (\Omega)^{\frac{n-p-1}{n-p}} \, .
%{\rm C}_p (\Omega)^{\frac{n-p-1}{n-p}} \left(\frac{n-p}{p-1}\right)^{\!p} \abs{\Sf^{n-1}} \, .
\end{equation}

\smallskip

\item The function $\Phi_p$ is differentiable at the regular values of $\phi$.

\smallskip

\item In terms of $\Phi_p$, the effective inequalities~\eqref{eq:crucial} correspond to 
\begin{equation}
\label{crucial_phi}
\Phi_p'(0)\,\leq\,0
\qquad\mbox{and}\qquad
\lim_{s\to +\infty}\Phi_p (s) \leq \Phi_p(0) \,.
\end{equation}
In fact it is easily seen that 
\begin{equation} 
\label{upfip}
\,\,\,\,\,\,\,\qquad \Phi_p(s) \, = \, U_p \Big( {\rm e}^{-\frac{(n-p)}{(p-1)(n-2)}s}  \Big) \quad\mbox{and}\quad 
\Phi_p'(s) 
\, = \, 
-\frac{(n-p)}{(p-1)(n-2)} \, U_p' \Big( {\rm e}^{-\frac{(n-p)}{(p-1)(n-2)}s}  \Big)
\end{equation}
whenever these objects are well defined.
\end{itemize}
The inequalities~\eqref{crucial_phi} are at the core of our analysis and will be deduced in Section~\ref{proof}, as consequences of our {\em effective monotonicity fomulas} (see Theorem~\ref{monotone1} and Theorem~\ref{monotone2}).

\section{Proof of the $L^p$-Minkowski Inequality}
\label{proof}

%%%%%%%%%%%%%%%%%%%%%%%%%%%%%%%%%%%%%%%%%%%%%%%%%%
%%%%%%%%%%%%%%%%%%%%%%%%%%%%%%%%%%%%%%%%%%%%%%%%%%

The aim of this section is to give a complete proof of Theorem~\ref{premink}, namely the $L^p$-Minkowski Inequality
\begin{equation*}
{\rm C}_p (\Omega)^{\frac{n-p-1}{n-p}}  \leq \,  \frac{1}{\,  \abs{\Sf^{n-1}} \,} \, \int\limits_{\partial \Omega} \,\left\vert \frac{\HH}{n-1}\right\vert^{p}\! \dd\sigma  \, .
\end{equation*}
In force of the discussion in Subsection~\ref{sub:level} (see also Subsection~\ref{completion} below for a fully detailed proof), it is sufficient to establish the validity of the inequalities~\eqref{eq:crucial} in their conformal version~\eqref{crucial_phi}
\begin{equation*}
%\label{crucial_phi}
\Phi_p'(0)\,\leq\,0
\qquad\mbox{and}\qquad
\Phi_p (+ \infty) = \!\lim_{s\to +\infty}\Phi_p (s) \, \leq \, \Phi_p(0) \,.
\end{equation*}

\bigskip

\noindent {\em Since all the computations of this section will be performed in the conformally related setting, the subscript $g$ will be dropped from the notations.}

\bigskip

%%%%%%%%%%%%%%%%%%%%%%%%%%%%%%%%%%%%%%%%%%%%%%%%%%

\subsection{First effective inequality: $\Phi_p'(0) \leq 0$.}

%%%%%%%%%%%%%%%%%%%%%%%%%%%%%%%%%%%%%%%%%%%%%%%%%%

For a given $1<p<n$, let us consider the vector field 
\begin{equation}
\label{X}
X \, = \, {\rm e}^{-\frac{(n-p)}{(n-2)(p-1)}\phi}\left\vert \nabla \phi\right\vert^{p-2}\Bigg(\nabla\left\vert\nabla \phi\right\vert + (p-2) \nabla^{\perp} |\nabla \phi|
\Bigg) \,.
\end{equation}
As it can be readily checked, at a regular value of $\phi$ one has that 
\begin{equation}
\label{der_fip}
 {\rm e}^{\! -\frac{(n-p)}{(n-2)(p-1)}s} \,\,   {\Phi_p'(s)} \,\, = \,\,  \frac{1}{p-1} \!\! \int\limits_{\{\phi = s\}}   \!\!\! \left\langle \! X \, \bigg\vert \, \frac{\nabla \phi}{\abs{\nabla \phi}}\right\rangle \, \dd\sigma \, .
\end{equation}
In the next lemma, we compute the divergence of $X$.

\begin{lemma}[Divergence of $X$]
\label{divwlemma}
For any $1< p <n$, let $g$ and $\phi$ be the solutions to~\eqref{prob_ex_rif} obtained by the solution to~\eqref{pb} through~\eqref{g} and~\eqref{phi} and let $X$ be the vector field defined in~\eqref{X}.
%, where $g$ and $\phi$ are the solutions to~\eqref{prob_ex_rif} obtained by the solution to~\eqref{pb} through~\eqref{g} and~\eqref{phi}
Then, the following identity holds at any point $x \in \R^n \setminus \overline{\Omega}$ such that $\abs{\nabla \phi}(x) \neq 0$.
\begin{equation}
\label{divX}
\dive X 
\,\,= \,\, {{\rm e}^{-\frac{(n-p)}{(n-2)(p-1)} \phi}}  \,\, Q \,\, \geq \,\,  0 \, , 
\end{equation}
where 
\begin{equation}
\label{Q}
Q \,\, = \,\,  \abs{\nabla \phi}^{p-3}\,  \Bigg\{ \, \abs{\nabla \phi}^2\left\vert \, \hh - \frac{\HH}{n-1} \, g^{\!\top} \, \right\vert^2 \!\! + \, (p-1) \, \bigabs{\nabla^\top\abs{\nabla\phi}}^2 \!\!+\, \frac{(p-1)^2}{n-1} \, 
\bigabs{\nabla^\perp\abs{\nabla\phi}}^2
%\bigg\langle\nabla\abs{\nabla\phi} \, \bigg\vert \, \frac{\nabla\phi}{\abs{\nabla\phi}}\bigg\rangle^2
\, \Bigg\},
\end{equation}
where $\hh$ and $\HH$ are respectively the second fundamental form and the mean curvature of the level sets of $\phi$ with respect to the unit normal $\nabla \phi /\abs{\nabla \phi}$.
\end{lemma}

\begin{proof} For the sake of clearness, we write 
$$
X = {\rm e}^{-\frac{(n-p)}{(n-2)(p-1)}\phi}( W + Z) \, ,
$$ 
where
\begin{equation*}
W = \left\vert \nabla \phi\right\vert^{p-2}\nabla\left\vert\nabla \phi\right\vert
\qquad
\hbox{and} 
\qquad 
Z = (p-2)\abs{\nabla \phi}^{p-2} \nabla^\perp |\nabla \phi|\, .
%
%\left\langle\nabla \left\vert\nabla \phi\right\vert \, \bigg\vert \, \frac{\nabla \phi}{\left\vert\nabla \phi\right\vert}\right\rangle  \frac{\nabla \phi}{\abs{\nabla \phi}}.
\end{equation*}
Using the same computation as in~\cite[Proposition 4.3]{Fog_Maz_Pin} with $q = p/(p-1)$ (in the notation of that paper), one finds that the divergence of $W$ is given by
\begin{equation}
\label{divx}
\begin{split}
\dive\, W\,\, = &\,\, \Big(\frac{n-p}{n-2}\Big) \, \big\langle W \, \vert \, \nabla\phi\big\rangle \, +
\\
+ &\,\, \abs{\nabla\phi}^{p-3}     
 \Bigg\{\abs{\nabla\nabla\phi}^2 - \bigabs{\nabla\abs{\nabla\phi}}^2 
%\\
%&
+\big(p-2\big)\left[ \,\,
\bigabs{\nabla^\perp\abs{\nabla\phi}}^2 
%\left\langle\nabla\abs{\nabla\phi} \, \bigg\vert \, \frac{\nabla\phi}{\abs{\nabla\phi}}\right\rangle^2
-\frac{\nabla\nabla \abs{\nabla\phi}(\nabla\phi, \nabla\phi)}{\abs{\nabla\phi}} \, \right]\Bigg\} .
\end{split}
\end{equation}
Plugging the Kato-type identity~\eqref{katof} in~\eqref{divx} and using the standard decomposition~\eqref{decomposition-intro}, 
we immediately get 
\begin{equation}
\label{divx2}
\begin{split}
\dive W\,\, = &\,\, \Big(\frac{n-p}{n-2}\Big) \, \big\langle W \, \vert \, \nabla\phi\big\rangle \,\, +
\\
+ &\,\, \abs{\nabla\phi}^{p-3}   \,   
 \Bigg\{
%
%\begin{split}
%\dive \, W = \frac{n-p}{n-2} \, & \big\langle W \, \vert \, \nabla\phi\big\rangle  
%+ \abs{\nabla \phi}^{p-3} \Bigg\{
\abs{\nabla \phi}^2\left\vert \, \hh - \frac{\HH}{n-1} \, g^{\!\top} \, \right\vert^2  - \big(p-2\big) \, \frac{\nabla\nabla\abs{\nabla\phi}(\nabla\phi, \nabla\phi)}{\abs{\nabla\phi}} \,\,  + \\
&\qquad\qquad\qquad\qquad \, + \,  \Big\vert \nabla^\top \abs{\nabla \phi} \Big\vert^2 + \left[\frac{(p-1)^2}{n-1} + \big(p-2\big)\right] 
\,
\Big\vert \nabla^\perp \abs{\nabla \phi} \Big\vert^2
%\left\langle{\nabla\abs{\nabla \phi}} \, \bigg\vert \, \frac{\nabla \phi}{\abs{\nabla \phi}}\right\rangle^2  
%\\
%&\qquad\qquad\qquad\qquad\qquad\!\! 
\Bigg\}.
\end{split}
\end{equation}
Let us now compute the divergence of $Z$. Clearly, by the $p$-harmonicity of $\phi$, we have
\[
\dive Z \,\, = \,\,  (p-2) \, \abs{\nabla \phi}^{p-2} \,\, \Bigg\langle \nabla \! \left(\bigg\langle\nabla \! \left\vert\nabla \phi\right\vert \,\, \bigg\vert \,\, \frac{\nabla \phi}{\, \left\vert\nabla \phi\right\vert^2}\bigg\rangle\right)  \Bigg\vert \,  \nabla \phi\Bigg\rangle \, .
\]
Expanding the right hand side and using the identity
\[
\frac{\nabla\nabla \phi \big( \nabla\abs{\nabla \phi}, \nabla \phi \big)}{\abs{\nabla \phi}} \, = \,  \Big\vert\nabla \abs{\nabla \phi} \Big\vert^2
\]
yield
\begin{equation}
\label{divez}
\begin{split}
\dive  Z \,\, = \,\, (p-2)\,\, \abs{\nabla \phi}^{p-3} \, \Bigg\{ \, \frac{\nabla\nabla\abs{\nabla\phi}(\nabla\phi, \nabla\phi)}{\abs{\nabla\phi}} 
%-  \left\langle{\nabla\abs{\nabla \phi}} \, \bigg\vert \, \frac{\nabla \phi}{\abs{\nabla \phi}}\right\rangle^2 
\, + \, \Big\vert \nabla^\top \abs{\nabla \phi} \Big\vert^2 \!\!- \, \Big\vert \nabla^\perp \abs{\nabla \phi} \Big\vert^2 \, \Bigg\} \, .
\end{split}
\end{equation}
Finally, combining~\eqref{divx2} and~\eqref{divez},
and observing that
\[
\left\langle X \, \vert \, \nabla \phi\right\rangle \, = \,\, {\rm e}^{\!-\frac{(n-p)}{(n-2)(p-1)}\phi} \, (p-1)\, \left\langle W \, \vert \, \nabla \phi \right\rangle  \, ,
\]
we arrive at 
\[
\dive X \, = \,\, {\rm e}^{\!-\frac{(n-p)}{(n-2)(p-1)}\phi}\left(\dive W + \, \dive Z \,  - \, \Big(\frac{n-p}{n-2}\Big) \left\langle W \, \vert \, \nabla \phi\right\rangle\right) \,= \,\,  {\rm e}^{\! -\frac{(n-p)}{(n-2)(p-1)}\phi} Q \, .
\]
This completes the proof of the lemma.
\end{proof}
In absence of critical points, the Divergence Theorem applied to the vector field $X$ on the open region $\{s < \phi < S\}$, with $0 < s <S$, easily yields the inequality
\begin{equation}
\label{ineq-liv2}
\int\limits_{\{\phi=s\}}  \!\!\! \left\langle \! X \, \bigg\vert \, \frac{\nabla \phi}{\abs{\nabla \phi}}\right\rangle \dd\sigma \,\,\, \leq  \!\!\int\limits_{\{\phi=S\}} \!\!\left\langle \! X \, \bigg\vert \, \frac{\nabla \phi}{\abs{\nabla \phi}}\right\rangle  \dd\sigma \,, 
\end{equation}
and in turns, thanks to~\eqref{der_fip}, the inequality~\eqref{ineq-liv} below. In presence of a possibly wild critical set, this direct argument is no longer working. Fortunately, some of the new ideas introduced in~\cite{Ago_Maz_3} to treat the same issues in the case of harmonic functions are exportable to the case of $p$-harmonic functions, where 
one does not know {\em a priori} that the critical set is $(n-1)$-negligible.
As a consequence, we are still able to provide an effective version of the considered monotonicity, showing that~\eqref{ineq-liv} is actually in force, provided $s$ is small enough and $S$ is large enough. The desired effective inequality $\Phi_p' (0) \leq 0$, will follow at once.

\begin{theorem}[Effective Monotonicity Formula -- I]
\label{monotone1}
For any $1< p <n$, let $g$ and $\phi$ be the solutions to~\eqref{prob_ex_rif} obtained by the solution to~\eqref{pb} through~\eqref{g} and~\eqref{phi} and let $0< s_p < S_p<+\infty$ be such that $\crit(\phi) \subset \{s_p < \phi < S_p \}$. Then, for every  $0\leq s \leq s_p < S_p \leq S$, the inequality  
\begin{equation}
\label{ineq-liv}
  \frac{\Phi_p'(s)}{  \,\, {\rm e}^{\frac{(n-p)}{(n-2)(p-1)}s} \,\, }\,\ \leq \,\, \frac{\Phi_p'(S)}{ \,\, {\rm e}^{\frac{(n-p)}{(n-2)(p-1)}S} \,\, }\,\, 
\end{equation}
holds true, where $\Phi_p$ is the function defined in~\eqref{phip}. In particular, one has that $\Phi_p'(0) \leq 0$.
%\begin{equation}
%\label{ineq-liv2}
%\int_{\{\phi=S\}}\left\langle X \, \bigg\vert \, \frac{\nabla \phi}{\abs{\nabla \phi}}\right\rangle  \dd\sigma - \int_{\{\phi=s\}}\left\langle X \, \bigg\vert \, \frac{\nabla \phi}{\abs{\nabla \phi}}\right\rangle \dd\sigma \geq 0
%\end{equation}
\end{theorem}
\begin{proof}
For a given $\epsilon >0$, we consider a smooth nonnegative cut-off-function $\chi : [0, +\infty) \rightarrow \R$, such that
\begin{equation}
\label{chi}
\begin{cases}
\,\,\chi(t)= 0 & \mbox{in} \,\, t < \frac{1}{2}\epsilon\, ,  
\\
\,\,\dot{\chi}(t) \geq 0 & \mbox{in}\,\, \frac{1}{2}\epsilon \leq t \leq \frac{3}{2}\epsilon \, ,
\\
\,\,\chi(t) =1 &\mbox{in} \,\, t > \frac{3}{2}\epsilon \, .
\end{cases}
\end{equation}
Since $\chi\big(\abs{\nabla \phi}\big) = 0$ on $\crit(\phi)$,
we can apply the Divergence Theorem to the smooth vector field 
\[
\widetilde{X} \,\, = \,\, \chi\big(\abs{\nabla \phi}\big)\,  X
\]
in the domain $\{s < \phi < S\}$.
Observe that, choosing $\epsilon$ small enough, we can make sure that $\chi\big(\abs{\nabla \phi}\big) = 1$ on $\{\phi =s \}$ and $\{\phi =S\}$, since $\crit(\phi) \subset \{ s_p <\phi<S_p\}$. Having this in mind, we compute
\begin{equation}
\label{divth1}
\begin{split}
\hspace{-1cm}\int\limits_{\{\phi=S\}} \!\!\!\! \left\langle \!X \,\bigg\vert \, \frac{\nabla \phi}{\abs{\nabla \phi}}\right\rangle \dd\sigma \,\,\,  &- \!\!\!\int\limits_{\{\phi=s\}} \!\!\!\! \left\langle \!X \,\bigg\vert \, \frac{\nabla \phi}{\abs{\nabla \phi}}\right\rangle   \dd\sigma
%\int_{\{\phi=s\}}\left\langle X \,\bigg\vert \, \frac{\nabla \phi}{\abs{\nabla \phi}}\right\rangle \dd\sigma 
\,\, = \!\!\!\!\!\!\int\limits_{\{s < \phi < S\}}\!\!\!\!\!\! \dive \widetilde{X} \, \dd\mu  \,\, = \\ 
&=\!\!\!\!\!\!\!\!\! \int\limits_{\{s < \phi < S\} \setminus U_{\epsilon/{2}}}\!\!\!\!\!\!\!\!\!\!\!\! \chi\big(\abs{\nabla \phi}\big)\,\, \dive X \, \dd\mu  \,\,\, + \!\!\!\!\!\!\!\! \int\limits_{U_{{3}\epsilon/{2}} \setminus U_{\epsilon/{2}}} \!\!\!\!\!\!\!\! \dot\chi\big(\abs{\nabla \phi}\big) \, \left\langle X \, \vert \, \nabla \abs{\nabla \phi}\right\rangle \dd\mu \, ,
\end{split}
\end{equation}
where in the last identity we have used the tubular neighbourhood of $\crit(\phi)$ defined for every $\delta>0$ as $U_\delta = \{\abs{\nabla \phi} \leq \delta\}$. In view of~\eqref{der_fip}, \eqref{divX}, \eqref{chi} and~\eqref{divth1}, the inequality~\eqref{ineq-liv} is proved if we show that $\left\langle X \, | \,  \nabla \abs{\nabla \phi}\right\rangle \geq 0$ on ${U_{{3}\epsilon/{2}} \setminus U_{\epsilon/{2}}}$.
On the other hand, a direct computation gives 
\begin{equation}
\begin{split}
\left\langle X\,  \big| \, \nabla \abs{\nabla \phi}\right\rangle \,\, &= \,\, {\rm e}^{-\frac{(n-p)}{(n-2)(p-1)} \phi} \,  \abs{\nabla \phi}^{p-2} \, \left[ \,\, \big\vert\nabla \abs{\nabla \phi}\big\vert^2 \! + (p-2) \big\vert\nabla^\perp \abs{\nabla \phi}\big\vert^2
%\left\langle\nabla \abs{\nabla \phi} \, \bigg\vert \, \frac{\nabla \phi}{\abs{\nabla \phi}}\right\rangle^2 
\,\, \right]  \, = \\
&= \,\, {\rm e}^{-\frac{(n-p)}{(n-2)(p-1)} \phi} \,  \abs{\nabla \phi}^{p-2} \, \left[ \,\, \big\vert\nabla^\top \abs{\nabla \phi}\big\vert^2 \! + (p-1) \big\vert\nabla^\perp \abs{\nabla \phi}\big\vert^2
%\left\langle\nabla \abs{\nabla \phi} \, \bigg\vert \, \frac{\nabla \phi}{\abs{\nabla \phi}}\right\rangle^2 
\,\, \right]  \, \geq \, 0 \, .
%&= {\rm e}^{-\frac{(n-p)}{(n-2)(p-1)}\phi} \, \abs{\nabla \phi}^{p-2} \left[ \Big\vert \nabla_T \abs{\nabla \phi} \Big\vert^2 + (p-1)\left\langle\nabla \abs{\nabla \phi} \, \bigg\vert \, \frac{\nabla \phi}{\abs{\nabla \phi}}\right\rangle^2 \right] \\
%&\geq 0,
\end{split}
\end{equation}
This completes the proof of the first part of the statement. It remains to show that $\Phi_p'(0) \leq 0$. From~\eqref{ineq-liv} it follows at once that, for every $S\geq S_p$, it holds
\begin{equation}
{\rm e}^{\frac{(n-p)}{(n-2)(p-1)}S}\, \Phi'_p(0) \,\, \leq\,\, \Phi'_p(S)  \, .
\end{equation}
Integrating both sides of the above inequality on an interval of the form $(S_p, S)$, with $S_p<S$, we obtain 
\begin{equation}
\frac{(n-2)(p-1)}{(n-p)} \, {\rm e}^{\frac{(n-p)}{(n-2)(p-1)}S}\, \Phi_p'(0) \, + \, \Phi_p (S_p) \,  - \, \frac{(n-2)(p-1)}{(n-p)} \, {\rm e}^{\frac{(n-p)}{(n-2)(p-1)}S_p} \, \Phi_p'(0) \,\, \leq \,\, \Phi_p(S) \, .
\end{equation}
If by contradiction, $\Phi_p'(0) > 0$, then, letting $S\to + \infty$ in the above inequality, we would deduce that $\Phi_p(S) \to +\infty$, against the boundedness of $\Phi_p$ discussed at the end of Subsection~\ref{conformalsetting}.
\end{proof}

%%%%%%%%%%%%%%%%%%%%%%%%%%%%%%%%%%%%%%%%%%%%%%%%%

\subsection{Second effective inequality: $\Phi_p(+\infty) \leq \Phi_p(0)$.}

%%%%%%%%%%%%%%%%%%%%%%%%%%%%%%%%%%%%%%%%%%%%%%%%%

As already observed several times, the presence of critical points and critical values possibly arranged in sets with full measure makes the full monotonicity not expectable in general. 
In fact, the lack of a sufficiently strong Sard-type 
property for the $p$-capacitary potentials prevents any kind of straightforward adaptation of the arguments presented in~\cite{Ago_Maz_3}
(it is worth mentioning though \cite{Cha_Mor_Pon,Cha_Mor_Nov_Pon},
where a generic non-fattening property is proved
for the level sets of $p$-harmonic functions). 
In other words, there is no hope for deducing the global inequality $\Phi_p(+\infty) \leq \Phi_p(0)$ from the pointwise inequality $\Phi_p'(s) \leq 0$ through integration, since the latter inequality may fail to be true -- or even well defined -- for too many values of $s \in [0, +\infty)$. To face the main difficulty of our program, we craft a new family of effective monotonicity formulas. For a given $1<p<n$ and a given $0 < \lambda < 1$, we consider the vector field 
\begin{equation}
\label{Ylambda}
Y_\lambda \,\, = \,\, \Big({\rm e}^{ \frac{(n-p)}{(n-2)(p-1)} \phi}  -   \lambda\Big)\,  X \, - \, \Big(\frac{n-p}{n-2} \Big) \, \abs{\nabla \phi}^{p-1} \nabla \phi \, ,
\end{equation}
where $X$ has been defined in~\eqref{X}.
It is convenient to observe that at a regular value of $\phi$ it holds
\begin{equation}
\label{fip_derfip}
\bigg(\frac{ \,\, {\rm e}^{\frac{(n-p)}{(n-2)(p-1)}s} - \lambda\,\, }{{\rm e}^{\frac{(n-p)}{(n-2)(p-1)}s}} \bigg) \, \Phi'_p(s) \, - \, \frac{(n-p)}{(n-2)(p-1)} \Phi_p (s)\,\, = \,\, \frac{1}{(p-1)} \!\!\int\limits_{\{\phi = s\}} \!\!\! \left\langle Y_\lambda \, \bigg\vert \, \frac{\nabla \phi}{\abs{\nabla \phi}} \right\rangle \dd\sigma \, .
\end{equation}
In the next lemma, we compute the divergence of $Y_\lambda$.

\begin{lemma}[Divergence of $Y_\lambda$]
\label{divYlambdalemma}
For any $1< p <n$ and any $0<\lambda<1$, let $g$ and $\phi$ be the solutions to~\eqref{prob_ex_rif} obtained by the solution to~\eqref{pb} through~\eqref{g} and~\eqref{phi} and let $Y_\lambda$ be the vector field defined in~\eqref{Ylambda}.
%, where $g$ and $\phi$ are the solutions to~\eqref{prob_ex_rif} obtained by the solution to~\eqref{pb} through~\eqref{g} and~\eqref{phi}
Then, the following identity holds at any point $x \in \R^n \setminus \overline{\Omega}$ such that $\abs{\nabla \phi}(x) \neq 0$
\begin{equation}
\label{divwalphaeq}
\dive Y_\lambda \,\, = \,\, \bigg(\frac{ \,\, {\rm e}^{\frac{(n-p)}{(n-2)(p-1)}\phi} - \lambda\,\, }{{\rm e}^{\frac{(n-p)}{(n-2)(p-1)}\phi}} \bigg) \, Q \,\, \geq \,\, 0 \, ,
%\left(e^{\frac{(n-p)}{(n-2)(p-1)}\phi} - \lambda\right)\frac{Q}{e^{\frac{(n-p)}{(n-2)(p-1)}\phi}} ,
\end{equation}
where $Q$ is the nonnegative quantity defined in~\eqref{Q}.
\end{lemma}

\begin{proof}
By the very definition of $Y_\lambda$, we have that 
\begin{equation}
\begin{split}
\label{diveYlambda1}
\dive Y_\lambda \,\, = &\,\, \left({\rm e}^{\frac{(n-p)}{(n-2)(p-1)}\phi} - \lambda\right) \, \dive X \,\, + \,\,  \frac{(n-p)}{(n-2)(p-1)}\,\,  {\rm e}^{\frac{(n-p)}{(n-2)(p-1)}\phi}\left\langle X \, \vert \, \nabla \phi\right\rangle \, +  \\
& \, - \Big( \frac{n-p}{n-2} \Big) \,\, \dive \big(\abs{\nabla \phi}^{p-1} \nabla \phi\big) \, .
\end{split}
\end{equation}
Using the definition~\eqref{X} of the vector field $X$, we compute
\[
\left\langle X \, \vert \, \nabla \phi \right\rangle \,\, = \,\, (p-1)\, {\rm e}^{-\frac{(n-p)}{(n-2)(p-1)}\phi} \, \abs{\nabla \phi}^{p-2} \,  \left\langle \nabla \abs{\nabla \phi} \, \vert \, \nabla \phi\right\rangle \, .
\]
Exploiting the $p$-harmonicity of $\phi$, we get
\[
\dive \big(\abs{\nabla \phi}^{p-1} \nabla \phi\big) \,\, = \,\,  \abs{\nabla \phi}^{p-2}\left\langle \nabla\abs{\nabla \phi} \, \vert \, {\nabla \phi}\right\rangle \, .
\]
We conclude that
\[
\dive Y_\lambda \,\, = \,\,  \Big({\rm e}^{\frac{(n-p)}{(n-2)(p-1)}\phi} - \lambda\Big) \, \dive X \,\, = \,\, \bigg(\frac{ \,\, {\rm e}^{\frac{(n-p)}{(n-2)(p-1)}\phi} - \lambda\,\, }{{\rm e}^{\frac{(n-p)}{(n-2)(p-1)}\phi}} \bigg) \, Q  \, ,
%\left(e^{\frac{(n-p)}{(n-2)(p-1)}\phi} - \lambda\right)\frac{Q}{e^{\frac{(n-p)}{(n-2)(p-1)}\phi}},
\]
where in the last equality we made use of the identity~\eqref{divX}.
\end{proof}

Again, in absence of critical points, the Divergence Theorem applied to the vector field $Y_\lambda$ on the open region $\{s < \phi < S\}$ easily yields the inequality
\begin{equation}
\label{monotone2f}
\int\limits_{\{\phi=s\}}  \!\!\! \left\langle \! Y_\lambda \, \bigg\vert \, \frac{\nabla \phi}{\abs{\nabla \phi}}\right\rangle \dd\sigma \,\,\, \leq  \!\!\int\limits_{\{\phi=S\}} \!\!\left\langle \! Y_\lambda \, \bigg\vert \, \frac{\nabla \phi}{\abs{\nabla \phi}}\right\rangle  \dd\sigma \,, 
\end{equation}
and in turns, thanks to~\eqref{fip_derfip}, the inequality~\eqref{mon_2} below. As usual, the difficult part is the treatment of the critical points. However, a quite surprising computation in the spirit of Theorem~\ref{monotone1} shows that it is always possible to deduce the second effective inequality.

\begin{theorem}[Effective Monotonicity Formula -- II]
\label{monotone2}
For any $1< p <n$, let $g$ and $\phi$ be the solutions to~\eqref{prob_ex_rif} obtained by the solution to~\eqref{pb} through~\eqref{g} and~\eqref{phi} and let $0< s_p < S_p<+\infty$ be such that $\crit(\phi) \subset \{s_p < \phi < S_p \}$. Then, for every $0<\lambda<1$ and every  $0\leq s \leq s_p < S_p \leq S$, the inequality 
%Let $Y_\lambda$ be defined by \eqref{Ylambda}. Then, for any $S > s > 0$ such that $\crit(\phi) \subset \{s< \phi < S\}$ we have
\begin{equation}
\label{mon_2}
\hspace{-0.2cm}\bigg(\frac{ \, {\rm e}^{\frac{(n-p) s}{(n-2)(p-1)}} - \lambda\,\, }{{\rm e}^{\frac{(n-p)s}{(n-2)(p-1)}}} \bigg) \, \Phi'_p(s) \, - \, \frac{(n-p) \, \Phi_p (s)}{(n-2)(p-1)}  \, \leq \, \bigg(\frac{ \, {\rm e}^{\frac{(n-p) S}{(n-2)(p-1)}} - \lambda\,\, }{{\rm e}^{\frac{(n-p)S}{(n-2)(p-1)}}} \bigg) \, \Phi'_p(S) \, - \, \frac{(n-p) \, \Phi_p (S) }{(n-2)(p-1)}  
%\label{monotone2f}
%\int_{\{\phi = S\}} \left\langle Y_\lambda \,\bigg\vert \, \frac{\nabla \phi}{\abs{\nabla \phi}} \right\rangle \dd\sigma - \int_{\{\phi = s\}} \left\langle Y_\lambda \, \bigg\vert \, \frac{\nabla \phi}{\abs{\nabla \phi}} \right\rangle \dd\sigma \geq 0. 
\end{equation}
holds true, where $\Phi_p$ is the function defined in~\eqref{phip}. In particular, one has that $\Phi_p(+\infty) \leq \Phi_p(0)$.
\end{theorem}

\begin{proof}
Let $\chi : [0, +\infty) \rightarrow \R$ be the  same smooth nonnegative cut-off function as in the 
proof of Theorem~\ref{monotone1}, so that the properties~\eqref{chi} are in force. To simplify the notation, let us also set 
\begin{equation}
\label{etalambda}
\eta_\lambda (\phi) \,\, = \,\, \frac{1}{ \,\, {\rm e}^{\frac{(n-p)}{(n-2)(p-1)}\phi} - \lambda\,\,}\, .
\end{equation}
Finally, let us consider the smooth vector field
\begin{equation}
\label{tildeY}
\widetilde{Y}_\lambda \,\,  = \,\,  \chi\big(\eta_\lambda(\phi) \, \abs{\nabla \phi}\big) \, Y_\lambda \, ,
\end{equation}
where $Y_\lambda$ has been defined in~\eqref{Ylambda}. Again, choosing $\epsilon$ small enough, we can suppose $\widetilde{Y}_\lambda = Y_\lambda$ on $\{\phi =s\}$ and $\{\phi = S\}$, with $s$ and $S$ as in the statement. Hence, applying the Divergence Theorem to the smooth vector field $\widetilde{Y}_\lambda$ on the region $\{ s< \phi <S\}$ gives
\begin{equation}
\label{divth2}
\begin{split}
\int\limits_{\{\phi=S\}} \!\!\!\! \left\langle \!Y_\lambda \,\bigg\vert \, \frac{\nabla \phi}{\abs{\nabla \phi}}\right\rangle &\dd\sigma \,\,  - \!\!\!\int\limits_{\{\phi=s\}} \!\!\!\! \left\langle \!Y_\lambda \,\bigg\vert \, \frac{\nabla \phi}{\abs{\nabla \phi}}\right\rangle   \dd\sigma
%\int_{\{\phi=s\}}\left\langle X \,\bigg\vert \, \frac{\nabla \phi}{\abs{\nabla \phi}}\right\rangle \dd\sigma 
\,\, = \!\!\!\!\!\!\int\limits_{\{s < \phi < S\}}\!\!\!\!\!\! \dive \widetilde{Y}_\lambda \, \dd\mu  \,\, = \\ 
&=\!\!\!\!\!\!\!\!\! \int\limits_{\{s < \phi < S\} \setminus U_{\epsilon/{2}}}\!\!\!\!\!\!\!\!\!\!\!\! \chi \big(\eta_\lambda(\phi) \, \abs{\nabla \phi}\big)\,\, \dive Y_\lambda \, \dd\mu  \,\,\, + \!\!\!\!\!\!\!\! \int\limits_{U_{{3}\epsilon/{2}} \setminus U_{\epsilon/{2}}} \!\!\!\!\!\!\!\! \dot\chi \big(\eta_\lambda(\phi) \, \abs{\nabla \phi}\big)
\, \Big\langle Y_\lambda \, \vert \, 
\nabla\big(\eta_\lambda(\phi)\abs{\nabla \phi}\big)\Big\rangle \dd\mu \, ,
\end{split}
\end{equation}
where this time the tubular neighbourhoods of $\crit(\phi)$ are defined, for every $\delta>0$, as $U_\delta = \{ \,  \eta_\lambda(\phi) \, \abs{\nabla \phi} \leq \delta \, \}$. Since, as observed in Lemma \ref{divYlambdalemma}, the divergence of $Y_\lambda$ is nonnegative on $\{s\leq \phi \leq S\} \setminus U_{\epsilon/{2}}$, where clearly $\abs{\nabla \phi} \neq 0$, the  inequality~\eqref{mon_2} is proved if we can show that 
\[
\left\langle Y_{\lambda} \, \Big\vert \, \nabla\Big(\eta_\lambda(\phi) \abs{\nabla \phi}\Big)\right\rangle \,\, \geq \,\,  0
\]
on ${U_{{3}\epsilon/{2}} \setminus U_{\epsilon/{2}}}$. A direct -- though not immediately evident --  computation, combined with the definition~\eqref{Ylambda} of $Y_\lambda$ yields
\begin{equation}
\label{alphaeta}
\begin{split}
\left\langle  \!Y_{\lambda} \, \Bigg\vert \, \nabla \!\left(\frac{\abs{\nabla \phi}}{{\rm e}^{\frac{(n-p)}{(n-2)(p-1)}\phi} \!\!-\lambda} \right)\right\rangle\,  =& \,\,\, {\rm e}^{-\frac{(n-p)}{(n-2)(p-1)} \phi} \,  \abs{\nabla \phi}^{p-2} \, \left[ \,\, \big\vert\nabla^\top \abs{\nabla \phi}\big\vert^2 \! + (p-1) \big\vert\nabla^\perp \abs{\nabla \phi}\big\vert^2
\,\, \right] \, + 
%\,e^{-\frac{(n-p)}{(n-2)(p-1)}\phi}\abs{\nabla \phi}^{p-2}\left[\Big\vert\nabla_T \abs{\nabla \phi}\Big\vert^2 + (p-1) \left\langle \nabla \abs{\nabla \phi} \,\bigg\vert \, \frac{\nabla \phi}{\abs{\nabla \phi}}\right\rangle^2\right] 
\\
&-2 \, \Big(\frac{n-p}{n-2}\Big) \, \eta_\lambda(\phi)\,\, \abs{\nabla \phi}^p \left\langle \nabla \abs{\nabla \phi} \, \bigg\vert \, \frac{\nabla \phi}{\abs{\nabla \phi}}\right\rangle \, +
\\
&+\Big(\frac{n-p}{n-2}\Big)^{\!2}\, \eta_\lambda^2(\phi)\,\,
\frac{\,{\rm e}^{\frac{(n-p)}{(n-2)(p-1)}\phi} \abs{\nabla \phi}^{p+2}}{p-1} \\
=&\,\,\, {\rm e}^{-\frac{(n-p)}{(n-2)(p-1)} \phi} \,  \abs{\nabla \phi}^{p-2} \,\, \big\vert\nabla^\top \abs{\nabla \phi}\big\vert^2 \, + \\
%\!\!\!\!\!\!=e^{-\frac{(n-p)}{(n-2)(p-1)}\phi}\abs{\nabla \phi}^{p-2}\Big\vert\nabla_T \abs{\nabla \phi}\Big\vert^2  \\
&+\Bigg[ \Big(\frac{n-p}{n-2}\Big) \, \eta_\lambda(\phi) \, \Bigg(\frac{{\rm e}^{\frac{(n-p)}{(n-2)(p-1)}\phi}\abs{\nabla \phi}^{{p+2}}}{(p-1)}\Bigg)^{\!\!1/2}\!\!\! - \\ 
& \qquad\quad - \,\, \left\langle \nabla \abs{\nabla \phi}\, \bigg\vert \, \frac{\nabla \phi}{\abs{\nabla \phi}}\right\rangle\Big({(p-1)} \,  {\rm e}^{-\frac{(n-p)}{(n-2)(p-1)}\phi} \, \abs{\nabla \phi}^{{p-2}}\Big)^{\!\!1/2} \, \Bigg]^2 
%\\
%&\!\!\!\!\!\!\geq \,  0 \, ,
\end{split}
\end{equation}
This completes the proof of the first part of the statement, since the rightmost hand side is manifestly nonnegative. It remains to show that $\Phi_p(+\infty) \leq \Phi_p(0)$. Applying the inequality~\eqref{mon_2} with $0< \lambda <1$, $s=0$ and $S_p \leq S$, we get
\begin{equation}
\label{passaggio}
\frac{(n-p)}{(n-2)(p-1)} \, \Big( \Phi_p (S) - \Phi_p(0)\Big)  \,\, \leq \,\, - \, (1-\lambda) \, \Phi_p'(0) \,  + \,  \Bigg(\frac{{\rm e}^{\frac{(n-p)}{(n-2)(p-1)}S} \!\! - \lambda}{{\rm e}^{\frac{(n-p)}{(n-2)(p-1)}S}}\Bigg) \, \Phi'_p(S) \, . 
\end{equation}
Observe now that \eqref{ineq-liv} holds also for $S_p < s < S$ (the cut-off argument is not even necessary in this case). Then, the very same reasoning employed to deduce that $\Phi_p'(0) \leq 0$ gives also $\Phi'_p (s) \leq 0$ for any $s > S_p$. In particular, $\Phi_p$ is a definitely bounded monotone function, and this implies  $\liminf_{S\to+\infty} \Phi_p'(S)\leq0$. Hence, passing to the (inferior) limit 
as $S\to+\infty$ in the above inequality yields
\[
\frac{(n-p)}{(n-2)(p-1)} \, \Big(\lim_{S\to +\infty} \!\!\Phi_p (S) 
- \Phi_p(0)\Big) \,\, \leq \,\,  - \, (1- \lambda) \, \Phi'_p(0) \, .
\]
Letting $\lambda \to 1^-$ on the right hand side leads to the second effective inequality $\Phi_p(+\infty) \leq \Phi_p(0)$.
\end{proof}

%%%%%%%%%%%%%%%%%%%%%%%%%%%%%%%%%%%%%%%%%%%%%%%%%%%

\subsection{Completion of the proof of Theorem~\ref{premink}}
\label{completion}

%%%%%%%%%%%%%%%%%%%%%%%%%%%%%%%%%%%%%%%%%%%%%%%%%%%

We are finally in the position to complete the proof of the $L^p$-Minkowski inequality, together with the related rigidity statement.

\begin{proof}[Proof of Theorem \ref{premink}]
To obtain inequality~\eqref{preminkf}, it is sufficient to detail the proof sketched in Subsection~\ref{sub:level}. As already observed in~\eqref{crucial_phi}, the effective inequalities obtained in Theorems~\ref{monotone1} and~\ref{monotone2} correspond to $U_p'(1) \geq 0$ and $U_p(0^+) \leq U_p(1)$, respectively. The first effective inequality implies that 
\[
\int\limits_{\partial \Omega}\Big(\frac{p-1}{n-p}\Big) \, {\abs{\D\log u}}^{p}\dd\sigma \, \,\leq \,\, \int\limits_{\partial \Omega}\abs{\D u}^{p-1}\frac{\HH}{n-1}\dd\sigma \, ,
\]
since a direct computation shows that 
\begin{equation}
U_p'(\tau) \,\, = \,\, \frac{1}{p-1} \, \tau^{-\frac{n-1}{n-p}}\int\limits_{\{u = \tau\}} \! \abs{\D u}^{p-1}
\left[\HH - \frac{(p-1)(n-1)}{(n-p)} \abs{\D \log u}\right] \dd\sigma \, .
\end{equation}
Applying the H\"older inequality to the above right hand side, with conjugate exponents $a = p/(p-1)$ and $b =p$, one is left with
\begin{equation}
\label{step1}
\int\limits_{\partial \Omega} \abs{\D u}^{p} \dd\sigma \,\,  \leq \,\, \Big(\frac{n-p}{p-1}\Big)^{p}  \int\limits_{\partial \Omega} \left\vert \frac{\HH}{n-1}\right\vert^{p} \dd\sigma \, .
\end{equation}
Using the second effective inequality $U_p (0^+) \leq U_p(1)$ in combination with~\eqref{eq:limup} we get 
\begin{equation}
\Big(\frac{n-p}{p-1}\Big)^{\!p} \,  \abs{\Sf^{n-1}}\,\, {\rm C}_p (\Omega)^{\frac{n-p-1}{n-p}}  \, = \,  \lim_{\tau \to 0^+} U_p(\tau) \,\, \leq \,\, U_p(1) \, = \, \int\limits_{\partial \Omega} \abs{\D u}^{p} \dd\sigma \, ,
\end{equation}
that combined with~\eqref{step1} gives the desired
\begin{equation}
{\rm C}_p (\Omega)^{\frac{n-p-1}{n-p}}  \leq \,  \frac{1}{\,  \abs{\Sf^{n-1}} \,} \, \int\limits_{\partial \Omega} \,\left\vert \frac{\HH}{n-1}\right\vert^{p}\! \dd\sigma  \, .
%\label{pre-mink}
%{\rm C}_p (\Omega)^{1 - \frac{1}{n-p}} \abs{\Sf^{n-1}} \,\, \leq \,\, \int\limits_{\partial \Omega} \left\vert \frac{\HH}{n-1}\right\vert^{p} \dd\sigma.
\end{equation}
{Assume now that equality holds in \eqref{premink}. Then, equality holds in \eqref{step1}, and consequently  $U_p'(1) = \Phi'_p (0) = 0$. Let $s^* \in (0, + \infty]$ be the first critical value of $\phi$.  A straightforward perusal of the proof of Theorem \ref{monotone1} shows that $\dive_{\!g} X = 0$ in $\{\phi \leq s\}$ for any $s < s^*$. By \eqref{divX} and \eqref{katof} we deduce that $\abs{\nabla \nabla \phi}_g = 0$ in this region. Then, a very standard argument (see e.g. the proof of \cite[Theorem 4.1 (i)]{Ago_Maz_1}) shows that $(\{\phi \leq s\}, g)$ is isometric to the cylinder $\left([0, s] \times \{\phi = 0\}, dt \otimes dt + g_{\{\phi = 0\}}\right)$, and that 
$\abs{\nabla \phi}_g$ equals a (positive) constant  in this region. The existence of a critical value $s^* < + \infty$ would thus contradict the continuity of $\abs{\nabla \phi}_g$, that follows from the $\mathscr{C}^{1}$-regularity of $p$-harmonic functions. Then, $\abs{\nabla \nabla \phi}_g = 0$ on the whole $\R^n \setminus \overline{\Omega}$, and then \cite[Theorem 4.1 (ii)]{Ago_Maz_1} implies that $\partial \Omega$ is a sphere.}
\end{proof}

%%%%%%%%%%%%%%%%%%%%%%%%%%%%%%%%%%%%%%%%%%%%%%%%%%%%%
%%%%%%%%%%%%%%%%%%%%%%%%%%%%%%%%%%%%%%%%%%%%%%%%%%%%%

\section{Proof of the Extended Minkowski Inequality}
\label{sec:min}

%%%%%%%%%%%%%%%%%%%%%%%%%%%%%%%%%%%%%%%%%%%%%%%%%%%%%
%%%%%%%%%%%%%%%%%%%%%%%%%%%%%%%%%%%%%%%%%%%%%%%%%%%%%

In this section we derive the Extended Minkowski Inequality~\eqref{minkowskif}
\begin{equation*}
%\label{minkowskif}
\left(\frac{|\pa \Omega^*|}{|\Sf^{n-1}|}\right)^{\!\!\frac{n-2}{n-1}} \!\!\! \leq \, \frac {1}{|\Sf^{n-1}|} \,\, \int\limits_{ \pa \Omega}   \left|\frac{\HH}{n-1}\right|  \, \dd \sigma \, ,
\end{equation*}
by letting $p \to 1^+$ in the $L^p$-Minkowski Inequality~\eqref{preminkf}. The main task here (see Theorem~\ref{limit-pcapth}) is to compute -- and characterise geometrically -- the limit of the variational $p$-capacity of a bounded set with smooth boundary. 
%Also in view of the applications proposed in Section~\ref{sec:cor}, 
As we are going to see, this limit turns out to be related to the {\em strictly outward minimising hull} of $\Omega$, a notion that plays a central role in the formulation of the weak Inverse Mean Curvature Flow introduced in~\cite{Hui_Ilm}.

%%%%%%%%%%%%%%%%%%
%%%%%%%%%%%%%%%%%%
\subsection{(Strictly) outward minimising sets and the strictly outward minimising hull}
%%%%%%%%%%%%%%%%%%
%%%%%%%%%%%%%%%%%%
The notion of {\em outward minimising sets} and {\em strictly outward minimising sets} are given in the context of sets with finite perimeter. 
We refer the reader to \cite{maggi} for a comprehensive treatment
of the basic notions that we are going to recall.
\begin{definition}[Outward minimising and strictly outward minimising sets]
\label{def:out}
Let $E \subset \R^n$ be a bounded measurable set with finite perimeter. We say that $E$ is \emph{outward minimising} if for any $F \subset \R^n$ with $E \subseteq F$ we have $\abs{\partial^*E} \leq \abs{\partial^*F}$, where by $\partial^*F$ we denote the reduced boundary of a set $F$ . We say that $E$ is \emph{strictly outward minimising} if it is outward minimising and any time $\abs{\partial^*E} = \abs{\partial^*F}$ for some $F \subset\R^n$ with $E \subseteq F$ we have $\abs{F \setminus E} = 0$.
\end{definition}  
It is easily seen that a bounded open set with finite perimeter is (strictly) outward minimising if and only if any measure zero modification of it is (strictly) outward minimising. In order to define appropriate representatives for these sets, we recall the definition of the \emph{measure theoretic interior} of a set $E$ with $\abs{E} < + \infty$ as the points of density $1$ for $E$, namely
\begin{equation}
\label{interior}
\text{Int} (E) = \left\{x \in\R^n \ \Big\vert \ \lim_{r \to 0^+} \frac{\abs{E \cap B(x, r)}}{\abs{B(x, r)}}=1\,\right\}.
\end{equation}
It follows from Lebesgue Differentiation Theorem (see \cite[Theorem 5.16]{maggi}) that 
\begin{equation}
\label{representative}
\abs{E \Delta \, \text{Int} (E)} = 0 \, .
\end{equation}
Importantly, 
a set with finite perimeter $E$ satisfies,
\begin{equation}
\label{topint}
\partial \,\text{Int} (E) = \overline{\partial^* E},
\end{equation}
that is, the topological boundary of the measure theoretic interior of a set with finite perimeter coincides with the closure of its reduced boundary. We address the reader to \cite[Theorem 10]{caraballo} for a proof of this nice property. We are now ready to define the {\em strictly outward minimising hull} of a set. As it can be checked, 
this concept essentially coincides with the one 
outlined in~\cite[p. 371]{Hui_Ilm}.
%The following is the definition, given in \cite[p. 371]{Hui_Ilm}, of the strictly outward minimising hull.

\begin{definition}[Strictly outward minimising hull]
\label{smh-def}
Let $\Omega \subset \R^n$ be a bounded open set with smooth boundary. We define the \emph{strictly outward minimising hull} of $\Omega$ as the measure theoretic interior of a set $E$ solving the minimisation problem
\begin{equation}
\label{def:somh}
\inf_{E \in \mathrm{SOMBE}(\Omega)} \abs{E} ,
\end{equation}
where 
\[
\mathrm{SOMBE} \, (\Omega) 
\,=\, 
\big\{E \, \vert \, \Omega \subseteq E \, \text{and} \,\, E \, \, \text{is bounded and strictly outward minimising}\big\}.
\]
\end{definition}
According to \cite{fogagnolo-mazzieri-hulls} (see also \cite{Hui_Ilm}), $\Omega^*$  is a solution of the area minimisation problem with obstacle $\Omega$, that is
\begin{equation}
\label{obstacle}
\abs{\partial^*\Omega^*} 
\,=\, 
\inf\big\{\
\abs{\partial^* F} \,\, \big\vert \,\, \Omega \subseteq F
\big\}. 
\end{equation}
We recall that the main result in~\cite{sternberg-williams}
(see also the comprehensive \cite[Theorem 1.3]{Hui_Ilm})
provides us with a regularity result for
any solution $E$ to \eqref{obstacle}
such that $\partial E = \overline{\partial^* E}$.
Note that $\Omega^*$ fulfils this requirement
(in view of~\eqref{topint} combined with the fact that $\rm{Int} (\Omega^*) = \Omega^*$, due to its very definition.
The topological boundary $\partial \Omega^*$ 
is thus equipped with the following regularity. 

\begin{theorem}[Regularity of the strictly outward minimising hull]
\label{regularity-hulls}
Let $\Omega \subset \R^n$ be a bounded set with smooth boundary. Then
\begin{itemize}
\item[(i)] $\partial \Omega^*$ is a $\mathscr{C}^{1, 1}$ hypersurface 
in a neighbourhood of any point in $\partial \Omega^* \cap \partial \Omega$.
\item[(ii)] $\partial \Omega^*$ is area minimising in $\partial \Omega^* \setminus \partial \Omega$. In particular there exists a singular set 
${\rm Sing} \subset \partial \Omega^* \setminus \partial\Omega$, 
with Hausdorff dimension at most $n- 8$, such that $\partial \Omega^* \setminus \partial \Omega$ is a real analytic hypersurface in a neighbourhood of any point in  $(\partial \Omega^* \setminus \partial \Omega) \setminus{\rm Sing}$.
\end{itemize}
\end{theorem}

An obvious consequence of this theorem is the fact that $\abs{\partial^* \Omega^*} = \abs{\partial \Omega^*}$.
\begin{remark}
\label{alt-def}
Since, as already pointed out, the boundary of $\Omega^*$ 
has least area among sets enclosing $\Omega$, we have that a bounded open set $\Omega \subset \R^n$ with smooth boundary is outward minimising if and only if  
\begin{equation}
\label{alt-def-formula}
\abs{\partial \Omega} = \abs{\partial \Omega^*}. 
\end{equation}
Observe that the inequality $\abs{\partial \Omega^*} \leq \abs{\partial \Omega}$ is automatically satisfied, due to \eqref{obstacle}.
\end{remark}
Since $\abs{\partial^* \Omega^*} = \abs{\partial \Omega^*}$ we can apply the nice interior approximation result \cite[Theorem 1.1]{schmidt} to $B(x, R) \setminus \overline{\Omega^*}$, with $\Omega^* \Subset B(x, R)$, to obtain the following exterior approximation result.
\begin{lemma}[Smooth exterior approximation of the strictly outward minimising hull]
\label{appr}
Let $\Omega\subset \R^n$ be a bounded open set with smooth boundary. Then, there exists a 
sequence of bounded sets $\{\Omega_k\}_{k \in \N}$ with smooth boundary such that 
\begin{equation}
\label{apprf}
 \Omega^* \subset \Omega_k, \quad \abs{\partial\Omega_k} \to \abs{\partial\Omega^*}.
 \end{equation}
%Observe also that, in this case, the mean curvature of $\partial \Omega$ is nonnegative, as a standard fist variation argument immediately shows.
\end{lemma}

%%%%

\subsection{Minimising hulls and $p$-capacities.}
Let $\Omega \subset \R^n$ be a bounded open set with smooth boundary. Recall from Definition~\ref{pcap} that for $1<p <n$ one has
\begin{equation}
\label{pcap-fin}
\capa_p(\Omega) \,  = \, \inf\left\{ \int_{\R^n} \abs{ \D f}^p \dd\mu \,\, \Big| \ f \geq \chi_\Omega, \,\, f \in \mathscr{C}^\infty_c(\R^n) \right\}.
\end{equation}
We can define, analogously, the $1$-capacity of a bounded open set with smooth boundary $\Omega$ as
\begin{equation}
\label{cap1-fin}
\capa_1(\Omega) \, = \, \inf\left\{ \int_{\R^n} \abs{ \D f} \dd\mu \,\, \Big| \ f \geq \chi_\Omega, \,\, f \in \mathscr{C}^\infty_c(\R^n) \right\}.
\end{equation}
The following result says that $\capa_1(\Omega)$
can indeed be recovered as the limit of
$\capa_p(\Omega)$, as $p \to 1^+$,
and that these quantities are also
related with the strictly 
outward minimising hull of $\Omega$.
\begin{theorem}
\label{limit-pcapth}
Let $\Omega \subset \R^n$ be a bounded open set with smooth boundary. Then,
\begin{equation}
\label{limit-pcap}
\lim_{p \to 1^+} \capa_p(\Omega) = \abs{\partial \Omega^*},
\end{equation}
where $\Omega^*$ is the strictly 
outward minimising hull of $\Omega$.
\end{theorem}
\begin{proof}
Let us first observe that for any $f \in \mathscr{C}^\infty_c(\R^n)$ with $f \geq \chi_\Omega$ we have, by co-area formula,
\begin{equation}
\label{1-ineq-pre}
\int\limits_{\R^n} \abs{\D f} \dd\sigma 
\,\geq\, 
\int_0^1 \abs{\{f = t\}} \dd t 
\,\geq\, 
\inf \Big\{ \abs{\partial E}\ \big|\ {\Omega} \subset E, \, \partial E \, \text{smooth}\Big\} 
\,\geq\, \abs{\partial \Omega^*},
\end{equation}
where  the last equality is due to Lemma \ref{appr}. In particular, taking the infimum over any such $f$, we get
\begin{equation}
\label{1-ineq}
\abs{\partial \Omega^*} \leq \capa_1(\Omega).
\end{equation}

We now prove that 
\begin{equation}
\label{liminf}
\capa_1(\Omega) \leq \liminf_{p \to 1^+} \capa_p(\Omega).
\end{equation}
This will be done by passing to the limit as $p\to 1^+$ in the inequality appearing in the proof of \cite[Theorem 3.2]{xu}, 
keeping track of the appearing constants
(which results in inequality \eqref{xu3} below). 
Namely, for every $f \in \mathscr{C}^\infty_c(\R^n)$ with $f \geq \chi_\Omega$
and any positive exponent $q$, the function $f^q$ is an admissible competitor in \eqref{pcap-fin} and \eqref{cap1-fin}. Then, 
by definition of the $1$-capacity and by 
H\"older inequality we have
\begin{equation}
\label{xu1}
\capa_1(\Omega) \leq \int\limits_{\R^n} \abs{\D f^q} \dd \mu = q \int\limits_{\R^n} f^{q-1} \abs{\D f} \dd\mu \leq q \left(\int_{\R^n}f^{{(q-1)}\frac{p}{p-1}} \dd\mu\right)^{{(p-1)}/{p}} \left(\int_{\R^n} \abs{\D f}^{p} \dd\mu\right)^{1/p}\!\!\!.
\end{equation}
Let now $q$ satisfy ${(q-1)}{p}/{(p-1)} = p^*$, where $p^*=pn/(n-p)$ is the Sobolev conjugate exponent of $p$, that is 
\begin{equation}
\label{q}
q = 1 + p^*\frac{(p-1)}{p}.
\end{equation}
Observe that with this choice $q > 1$. Then, we obtain, applying the Sobolev inequality to the first integral in the right hand side of \eqref{xu1},
\begin{equation}
\label{xu2}
\capa_1(\Omega) 
\,\leq q \,
{\rm T}_{n, p}^{\,p^*(p-1)/p}
\left(\,\,
\int\limits_{\R^n} \abs{\D f}^{p} \dd\mu\right)^{\!\!p^*(p-1)/p^2 + 1/p} \!\!\!\!\!\!\!\!\!\!\!\!\! =  
\,\,\, q \,\, {\rm T}_{n, p}^{\,q-1}
\left(\,\,
\int\limits_{\R^n} \abs{\D f}^{p} \dd\mu\right)^{\!\!(n-1)/(n-p)},
\end{equation}
where ${\rm T}_{n, p}$ is Talenti's best constant in the Sobolev inequality, 
obtained in \cite{talenti}. 
We recall that the precise value of such constant is 
\begin{equation}
{\rm T}_{n, p} = \frac{1}{\pi^{1/2} n^{1/p}}\left(\frac{p-1}{n-p}\right)^{(p-1)/p}\left[\frac{\Gamma(1 + n/2)\Gamma(n)}{\Gamma(n/p)\Gamma( 1 + n - n/p)}\right]^{1/n},
\end{equation}
where $\Gamma$ is Euler's Gamma function.
Taking the infimum over any $f \in \mathscr{C}^{\infty}_c(\R^n)$ with $f \geq \chi_{\Omega}$ in \eqref{xu2},
we obtain
\begin{equation}
\label{xu3}
\capa_1(\Omega) 
\,\leq\, 
q \,{\rm T}_{n, p}^{\,q-1} \capa_p(\Omega)^{(n-1)/(n-p)}.
\end{equation}
As $p \to 1^+$, one can check that ${\rm T}_{n, p}$ converges to a positive constant, precisely (compare with \cite[p. 355]{talenti})
\[
\lim_{p\to 1^+}{\rm T}_{n, p} 
\,=\, 
\left(\frac{1}{\abs{\Sf^{n-1}}}\right)^{\!1/n},
\]
and this implies 
\begin{equation}
\label{lim-tal}
\lim_{p\to 1^+}{\rm T}_{n, p}^{\,q-1} = 1.
\end{equation}
Note that $q\to1$, as $p \to 1^+$, in view of \eqref{q}.
In turn, passing to the limit in \eqref{xu3}, we get
\eqref{liminf}.

We are left to prove the inequality
\begin{equation}
\label{limsup}
\limsup_{p \to 1^+} \capa_p(\Omega) \leq \abs{\partial \Omega^*}.
\end{equation}
Let $E$ be any open and bounded set in $\R^n$ with smooth boundary such that $\Omega \subset E$. Define, for $x \in \R^n$, the function $d_E (x) = \text{dist}(x, E)$. 
Moreover, let us introduce 
a smooth cut-off function $\chi_\epsilon$ fulfilling
\begin{equation}
%\label{chi}
\begin{cases}
\,\,\chi_\epsilon(t)= 1 & \mbox{in} \,\, t < \epsilon,  
\\
\,\,-\frac{1}{\epsilon} <\dot{\chi}(t) < 0 & \mbox{in}\,\, \epsilon \leq t \leq {2}\epsilon
\\
\,\,\chi_\epsilon(t) =0 &\mbox{in} \,\, t > {2}\epsilon,
\end{cases}
\end{equation}
and let us set $\eta_\epsilon(x) = \chi_\epsilon (d_E (x))$. Choosing $\epsilon$ small enough, it is easily seen, by the regularity of $d_E$ in a neighbourhood of $E$ (see \cite[Lemma 14.6]{Gil_Tru_book}), that the function $\eta_\epsilon$ is an admissible competitor in \eqref{pcap-fin}  and \eqref{cap1-fin}. Then,
\[
\capa_p(\Omega) \leq \int\limits_{\R^n} \abs{\D \eta_\epsilon}^p \dd\mu
\]
for any $p\geq 1$.
Letting $p \to 1^+$, we get
\[
\limsup_{p \to 1^+} \capa_p(\Omega) \leq \int\limits_{\R^n} \abs{\D \eta_\epsilon} \dd\mu = \int\limits_\epsilon^{2\epsilon} \abs{\dot{\chi_\epsilon}(t)} \left\vert\{d_E = t\}\right\vert \dd t,
\]
where in the last equality we applied the coarea formula combined with the fact that $\abs{\D d_E} = 1$ in a neighbourhood of $E$.
By the Mean Value Theorem, there exist $r_\epsilon \in (\epsilon, 2\epsilon)$ such that the above right hand side satisfies
\[
\int\limits_\epsilon^{2\epsilon} \abs{\dot{\chi_\epsilon}(t)} \left\vert\{d_E = t\}\right\vert \dd t = \epsilon \abs{\dot\chi_\epsilon(r_\epsilon)}\abs{\{d_E = r_\epsilon\}} <\abs{\{d_E = r_\epsilon\}},
\]
where the last inequality is due to the second condition in \eqref{chi}. Since,
as $r_\epsilon \to 0^+$, we clearly have
\[
\abs{\{d_E = r_\epsilon\}} \to \abs{\partial E},
\]
we conclude that
\[
\limsup_{p \to 1^+} \capa_p(\Omega) \leq \abs{\partial E}
\]
for any bounded open set $E$ with smooth boundary containing $\Omega$.
In particular, considering a sequence of bounded sets $\{\Omega_k\}_{k \in \N}$ with smooth boundary containing $\Omega^*$ and with $\abs{\partial \Omega_k} \to \abs{\partial \Omega^*}$ as $k \to \infty$, provided in Lemma \ref{appr}, we get \eqref{limsup}.
The inequalities \eqref{1-ineq}, \eqref{liminf} and \eqref{limsup} combine as
\[
\abs{\partial \Omega^*} \leq \capa_1 (\Omega) \leq \liminf_{p \to 1^+} \capa_p (\Omega) \leq \limsup_{p \to 1^+} \capa_p (\Omega) \leq \abs{\partial \Omega^*},
\]
completing the proof.
\end{proof}

%%%%%%%%%%%%%%%%%%%%%%%%%%%%%%%%%%%%%%%%%%%%%%%%%%

\subsection{Proof of Theorem~\ref{minkowski} and  Corollay~\ref{meanconvexki}}

%%%%%%%%%%%%%%%%%%%%%%%%%%%%%%%%%%%%%%%%%%%%%%%%%%

We are now in the position to prove  Theorem~\ref{minkowski} together with its Corollay~\ref{meanconvexki}.
\begin{proof}[Proof of Theorem \ref{minkowski}]
It suffices to pass to the limit as $p \to 1^+$ in \eqref{preminkf}, that is
\begin{equation}
\label{preminkf-pf}
{\rm C}_p (\Omega)^{\frac{n-p-1}{n-p}}  \leq \,  \frac{1}{\,  \abs{\Sf^{n-1}} \,}\int\limits_{\partial \Omega} \left\vert \frac{\HH}{n-1}\right\vert^{p}\! \dd\sigma  \, .
\end{equation}
Indeed, recalling the relation between $p$-capacity and normalised $p$-capacity given in Definition \ref{pcap},
Theorem \ref{limit-pcapth} shows that the left hand side of the above inequality behaves as  
\[
\lim_{p \to 1^+} {\rm C}_p(\Omega)^{\frac{n-p-1}{n-p}} = \left(\frac{\abs{\partial \Omega^*}}{\abs{\Sf^{n-1}}}\right)^{\frac{n-2}{n-1}},
\]
while the right-hand side of \eqref{preminkf-pf} is immediately seen to converge to the right hand side of \eqref{minkowskif}.
Spheres show the optimality of the estimate since their mean curvature is given by $(n-1)/R$, where $R$ is the radius of the ball they enclose.
\end{proof}

\begin{proof}[Proof Corollary \ref{meanconvexki}]
Inequality \eqref{meanconvexkif} immediately follows from the fact that outward minimising sets with smooth boundary satisfy $\abs{\partial \Omega^*} = \abs{\partial \Omega}$ and the mean curvature of their boundaries is nonnegative (see Remark \ref{alt-def}).

We are left to consider the equality case in \eqref{meanconvexkif} for some strictly outward minimising set with smooth and strictly mean-convex boundary. To this aim, let $\{\partial \Omega_t\}_{t \in [0, T)}$ be the evolution of $\pa \Om$ under smooth IMCF, up to some $T > 0$. %with initial datum $\partial \Omega$.
By~\cite[Lemma 2.4]{Hui_Ilm}, the weak IMCF starting at $\partial \Omega$ 
coincides with the smooth flow $\{\Omega_t\}_{t \in [0, T^*)}$,
for some  $0<T^*\leq T$. In particular, by~\cite[Lemma 1.4]{Hui_Ilm}, $\Omega_t$ is strictly outward minimising and strictly mean-convex for every $t \in [0, T^*)$, and then~\eqref{meanconvexkif} holds for every $\partial\Omega_t$ with $t \in [0, T^*)$.
We can then define, for $t \in [0, T^*)$, the monotonic quantity already discussed in Subsection~\ref{smuteuic}, namely
\begin{equation}
\label{q-function}
 \mathcal{Q} (t) \, = \, 
|\pa\Om_t|^{-\frac{n-2}{n-1}} \!\int\limits_{\pa\Om_t} \!\HH \,\dd \sigma\, .
%\qquad \hbox{with} \,\,\,\, \Sigma_0 = \pa \Omega.
\end{equation}
Observe that inequality~\eqref{meanconvexkif} is equivalent to $\mathcal{Q}(0) \geq (n-1) \abs{\Sf^{n-1}}^{1/(n-1)}$, and assuming equality in \eqref{meanconvexkif} is equivalent to $\mathcal{Q}(0) = (n-1)\abs{\Sf^{n-1}}^{1/(n-1)}$. 
By the smoothness of the flow, the function $\mathcal{Q}(t)$ is differentiable for $t \in [0, T)$, and then a straightforward computation involving the standard evolution equations provided e.g. in \cite[Theorem 3.2]{Huis_Pold} show that
\begin{equation}
\label{q-der}
\mathcal{Q}'(0) = - \abs{\partial \Omega}^{-\frac{n-2}{n-1}} \int\limits_{\partial \Omega} \frac{\abs{\mathring{\hh}}^2}\HH \leq 0 \,.
\end{equation}
However, since we assumed 
$\mathcal{Q}(0) = (n-1)\abs{\Sf^{n-1}}^{1/(n-1)}$, 
the strict inequality $\mathcal{Q}'(0) < 0$ would imply 
$\mathcal{Q}(t) < (n-1)\abs{\Sf^{n-1}}^{1/(n-1)}$ 
for some $t \in (0, T^*)$, which is equivalent to contradict \eqref{meanconvexkif} for some outward minimising $\Omega_t$ with strictly mean-convex boundary. Then $\mathcal{Q}'(0) = 0$ and in turn, by formula \eqref{q-der}, 
$\partial \Omega$ is totally umbilical. 
Therefore, $\partial \Omega$ must be a sphere.
\end{proof}

\section{Optimal nearly umbilical estimates for outward minimising sets}
\label{sec:cor}

This section is devoted to the proof of an optimal version of the well celebrated De Lellis-M\"uller nearly umbilical estimates for outward minimising domains.

\begin{theorem}[Optimal Nearly Umbilical Estimate]
\label{delellis-muller}
If $\Omega\subset \R^3$ is a bounded outward minimising open domain with 
smooth boundary, then 
\begin{equation}
\label{delellis-mullerf}
\int\limits_{\partial \Omega}\left| \, \hh\,  - \,  \frac{\overline{\HH}}{2} \, g_{\partial\Omega} \, \right|^2
\!\!{\rm d}\sigma 
\,\,\leq\,\, 
2 \int\limits_{\partial \Omega}
\big\vert\mathring{\hh}\big\vert^2 
\,{\rm d}\sigma \, ,
\end{equation}
where $g_{\partial\Omega}$ is the metric induced on $\partial \Omega$ by the Euclidean metric of $\R^3$, and
\begin{equation}
\label{Hmedia}
\overline{\HH} 
\,=\, 
\fint\limits_{\partial \Omega} \HH\dd\sigma \, ,
\qquad\qquad
\mathring{\hh} 
\,=\,
\hh - \frac{{\HH}}{2}g_{\partial\Omega} \, .
\end{equation}
Moreover, the equality is achieved in~\eqref{delellis-mullerf} by some {\em strictly mean-convex} and {\em strictly outward minimising} domain $\Omega$ if and only if $\Omega$ is isometric to a round ball.
\end{theorem}
A first main tool we are going to use in order to deduce Theorem \ref{delellis-muller} from Theorem \ref{minkowski} is the classical Gauss' equation for surfaces in $\R^3$, yielding
\begin{equation}
\label{gauss}
\RR_{\partial \Omega} = \HH^2 - \abs{\hh^2},
\end{equation}
where ${\rm R}_{\partial \Omega}$ is the scalar curvature of $\partial \Omega$ computed with respect to the metric $g_{\partial \Omega}$ induced on it by the Euclidean metric of $\R^3$.
A second main tool we need to recall is the famous Gauss-Bonnet formula, stating that
\begin{equation}
\label{gauss-bonnet}
\int\limits_{\partial \Omega} \RR_{\partial \Omega}\dd\sigma = 4\pi \chi({\partial \Omega}),
\end{equation}
where $\chi({\partial \Omega})$ is the Euler characteristic of the surface $\partial \Omega$.

We are finally going to show how the Minkowski inequality \eqref{minkowski} combined with these basic identities in differential geometry give the optimal nearly umbilical estimate \eqref{delellis-mullerf}.
\begin{proof}[Proof of Theorem \ref{delellis-muller}]
Expanding the squares, it is straightforwardly seen that \eqref{delellis-mullerf} is equivalent to 
\begin{equation}
\label{passagio-delellimuller}
\int\limits_{\partial \Omega}\Big(\abs{\HH}^2 - \abs{\hh}^2\Big)\dd\sigma \leq \frac{\overline{\HH}^2}{2}\abs{\partial \Omega}.
\end{equation}
Invoking Gauss' equation \eqref{gauss} and Gauss-Bonnet formula \eqref{gauss-bonnet}, we then obtain that \eqref{delellis-mullerf} is equivalent to 
\begin{equation}
\label{gaussin}
8\pi\chi(\partial \Omega) = 2\int\limits_{\partial \Omega} \RR_{\partial \Omega}\dd\sigma = 2\int\limits_{\partial \Omega}\Big(\abs{\HH}^2 - \abs{\hh}^2\Big)\dd\sigma \leq {\overline{\HH}^2}\abs{\partial \Omega},
\end{equation}
that is, to
\begin{equation}
\label{gaussin1}
\sqrt{{2\pi\chi(\partial \Omega)}\abs{\partial \Omega}}\leq \int\limits_{\partial \Omega} \frac{\HH}{2}\dd\sigma.
\end{equation}
Since obviously $\chi(\partial \Omega)\leq 2$, the inequality \eqref{gaussin1} follows from the Minkowski inequality \eqref{meanconvexkif}.

Assume now that equality holds for some outward minimising set $\Omega$ with smooth and strictly mean-convex boundary. Let $\{\partial \Omega_t\}_{t \in [0, T)}$ be evolving by IMCF with initial datum $\partial \Omega$. 
By \cite[Lemma 2.4]{Hui_Ilm}, the weak IMCF $\{E_t\}_{t \in [0, \infty)}$ starting at $\partial \Omega$ coincides with $\Omega_t$ for $t \in [0, T^*)$, for some $T^*$ possibly smaller than $T$. In particular, $\Omega_t$  is outward minimising and strictly mean-convex for any $t \in [0, T^*)$, for some $T^* > 0$, and then \eqref{delellis-mullerf} holds for $\partial\Omega_t$ for any $t \in [0, T^*)$.
We can then define, for $T \in [0, T^*)$,  the quantity
\begin{equation}
\label{perez-quantity}
\mathcal{P}(t) = \int\limits_{\partial \Omega_t}\abs{\mathring{\hh}}^2 \dd\sigma - \frac{1}{2}\int\limits_{\partial \Omega_t}\left(\HH - \frac{1}{\abs{\partial \Omega_t}}\int_{\partial \Omega_t}\HH \right)^2 \dd\sigma,
\end{equation}
introduced in  \cite[Chapter 3]{perez-thesis}.
Observe that inequality \eqref{delellis-mullerf} is equivalent to $\mathcal P(0) \geq 0$, assuming equality in \eqref{delellis-mullerf} is equivalent to $\mathcal{P}(0) = 0$. 
By the smoothness of the flow, the function $\mathcal{P}(t)$ is differentiable for $t \in [0, T)$, and then \cite[Lemma 3.4]{perez-thesis} yields
\begin{equation}
\label{perez-derivative}
\mathcal{P}'(0) = - \overline{\HH}\int_{\partial \Omega} \frac{\abs{\mathring{\hh}}^2}{\HH} \dd\sigma \leq 0
\end{equation}
However, since we assumed $\mathcal{P}(0) = 0$, $\mathcal{P}'(0) < 0$ would imply $\mathcal{P}(t) < 0$ for some $t \in (0, T^*)$ that is equivalent to falsify \eqref{delellis-mullerf} for some outward minimising $\Omega_t$ with strictly mean-convex boundary. Then $\mathcal{P}'(0) = 0$, and by formula \eqref{perez-derivative} this means that $\partial \Omega$ is totally umbilical, thus a sphere.
\end{proof}
Inequality \eqref{gaussin1} in the above proof, that we just showed to be equivalent to the nearly umbilical estimate \eqref{delellis-mullerf}, actually coincides with the Minkowski inequality for mean-convex hypersurfaces \eqref{meanconvexkif} if $\chi(\partial \Omega) = 2$, that is, if $\partial \Omega$ is diffeomorphic to a sphere. We want to show, with the following easy proposition, that such a diffeomorphism exists each time the right-hand side of \eqref{delellis-mullerf} is smaller than $16\pi$. 

\begin{proposition}
\label{lemma_8pi}
Let $\Omega\subset \R^3$ be a bounded open set with smooth and 
\emph{mean-convex} boundary. If
\begin{equation}
\label{boundh0}
\int\limits_{\partial \Omega} \left\vert \hh - \frac{{\HH}}{2}g_{\partial\Omega}\right\vert^2 \dd\sigma \leq 8\pi
\end{equation}
then $\partial \Omega$ is diffeomorphic to a sphere. In particular, if \eqref{boundh0} holds, then the Minkowski inequality is equivalent to the optimal nearly umbilical estimate.
\end{proposition}

\begin{proof}
If \eqref{boundh0} holds, we have, by the classical Willmore inequality \cite{Willmore}
\[
\int\limits_{\partial \Omega} \left\vert \hh - \frac{{\HH}}{2}g_{\partial\Omega}\right\vert^2 \dd\sigma \leq 8\pi \leq \frac{1}{2}\int\limits \HH^2 \dd\sigma,
\]
that implies
\begin{equation}
\label{willaplied}
\int\limits_{\partial \Omega}\Big(\abs{\HH}^2 - \abs{\hh}^2\Big)\dd\sigma \geq 0.
\end{equation}
Moreover, since equality is attained in the Willmore inequality if and only if $\partial \Omega$ is isometric to a sphere with the round metric, the same rigidity statement holds if equality is attained in \eqref{willaplied}. On the other hand, applying the Gauss' equation \eqref{gauss} and the Gauss-Bonnet formula \eqref{gauss-bonnet}, \eqref{willaplied} is equivalent to
\begin{equation}
\label{gaussbonnet-applied}
\chi(\partial \Omega) \geq 0.
\end{equation}
If $\chi(\partial \Omega)=0$ then by the characterization of the equality case in the Willmore inequality $\partial \Omega$ would be even isometric to a sphere, and this is a contradiction. Then $\chi(\partial \Omega)=2$, as claimed.
\end{proof}

%%%%%%%%%%%%%%%%%%%%%%%%%%%%%%%%%%%%%%%%%%%%%%%
%%%%%%%%%%%%%%%%%%%%%%%%%%%%%%%%%%%%%%%%%%%%%%%

\subsection*{Acknowledgements}
\emph{The authors are grateful to G.~De Philippis, N.~Fusco, G.~Huisken, A.~Malchiodi, C.~Mantegazza, L.~Mari, J.~Metzger,  J.~Scheuer and F.~Schulze for their interest in the present work, as well as for useful 
comments and discussions during the preparation of the manuscript. 
A special thank goes to L.~Benatti and F.~Oronzio for
the careful reading of this paper.
The authors are members of Gruppo Nazionale per l'Analisi Matematica, 
la Probabilit\`a e le loro Applicazioni (GNAMPA), which is part of the 
Istituto Nazionale di Alta Matematica (INdAM),
and they are partially funded by the GNAMPA project 
``Aspetti geometrici in teoria del potenziale 
lineare e nonlineare''.
}

%%%%%%%%%%%%%%%%%%%%%%%%%%%%%%%%%%%%%%%%%%%%%%%
%%%%%%%%%%%%%%%%%%%%%%%%%%%%%%%%%%%%%%%%%%%%%%%

\end{document}